\documentclass{article}

\usepackage{microtype}
\usepackage{graphicx}
\usepackage{subcaption}
\usepackage{booktabs} 
\usepackage{siunitx} 

\usepackage{hyperref}

\usepackage[preprint]{icml2026}

\usepackage{amsmath}
\usepackage{amssymb}
\usepackage{mathtools}
\usepackage{amsthm}

\DeclareMathOperator*{\argmax}{arg\,max}

\newcommand{\dJFB}{d_{x}^{JFB}}
\newcommand{\dJFBb}{d_{x_b}^{JFB}}
\newcommand{\dJFBxi}{d_{\xi_j}^{JFB}}

\newcommand{\R}{\mathbb{R}}

\newcommand{\Jcal}{\mathcal{J}}
\newcommand{\T}{^{\top}}
\newcommand{\invT}{^{-\top}}
\newcommand{\norm}[1]{\left\lVert#1\right\rVert}
\newcommand{\inner}[2]{\left\langle #1, #2 \right\rangle}

\usepackage[capitalize,noabbrev]{cleveref}

\theoremstyle{plain}
\newtheorem{theorem}{Theorem}[section]

\newtheorem{lemma}[theorem]{Lemma}
\newtheorem{corollary}[theorem]{Corollary}
\newtheorem{definition}[theorem]{Definition}
\newtheorem{assumption}[theorem]{Assumption}
\newtheorem{remark}[theorem]{Remark}

\usepackage[textsize=tiny]{todonotes}

\icmltitlerunning{On the Convergence of JFB for Implicit Optimal Control}

\begin{document}

\twocolumn[
  \icmltitle{On the Convergence of Jacobian-Free Backpropagation for  Optimal Control Problems with Implicit Hamiltonians}

  \begin{icmlauthorlist}
    \icmlauthor{Eric Gelphman}{CSM}
    \icmlauthor{Deepanshu Verma}{CLEM}
    \icmlauthor{Nicole Tianjiao Yang}{UTK}
    \icmlauthor{Stanley Osher}{UCLA}
    \icmlauthor{Samy Wu Fung}{CSM}
  \end{icmlauthorlist}

  \icmlaffiliation{CSM}{Department of Applied Mathematics and Statistics, Colorado School of Mines}
  \icmlaffiliation{CLEM}{Department of Mathematical Sciences, Clemson University}
  \icmlaffiliation{UTK}{Department of Mathematics, University of Tennessee Knoxville}
  \icmlaffiliation{UCLA}{Department of Mathematics, University of California Los Angeles}

  \icmlcorrespondingauthor{Samy Wu Fung}{swufung@mines.edu}

  \vskip 0.3in
]

\printAffiliationsAndNotice{}  

\begin{abstract}
Optimal feedback control with implicit Hamiltonians poses a fundamental challenge for learning-based value function methods due to the absence of closed-form optimal control laws. 
Recent work~\cite{gelphman2025end} introduced an implicit deep learning approach using Jacobian-Free Backpropagation (JFB) to address this setting, but only established sample-wise descent guarantees. In this paper, we establish convergence guarantees for JFB in the stochastic minibatch setting, showing that the resulting updates converge to stationary points of the expected optimal control objective. We further demonstrate scalability on substantially higher-dimensional problems, including multi-agent optimal consumption and swarm-based quadrotor and bicycle control. Together, our results provide both theoretical justification and empirical evidence for using JFB in high-dimensional optimal control with implicit Hamiltonians.
\end{abstract}

\section{Introduction}
\label{sec:intro}
    We aim to generate semi-global feedback controllers for high-dimensional control problems of the form 
    \begin{equation}
        \begin{split}
        \min_{u \in U} \int_0^T L(s, z_{x}, u) ds + G(z_{x}(T)), \\ \text{subject to} \quad \dot{z}_{x} = f(t, z_{x},u), \;\; z_{x}(0) = x,
        \end{split}
        \label{eq:originalOC}
    \end{equation}
    where $z_x\in \mathbb{R}^n$   is the state trajectory with dynamics $f$ and initial condition $x$, $u(t)\in U\subset \R^m$ is the control input, $L$ is the running cost, and $G$ is the terminal cost. The subscript in  $z_{x}$ denotes the dependence of the state on the initial condition $x$. 
    
    The Pontryagin Maximum Principle (PMP)~\cite{pontryagin2018mathematical, kopp1962pontryagin} provides a first-order characterization of optimal solutions to continuous-time control problems by introducing an adjoint (costate) variable that couples the system dynamics and the objective. In particular, PMP associates to the control problem a generalized Hamiltonian of the form
    \begin{equation}
    \mathcal{H}(t,z_x,p_x,u)
    =
    - \langle p_x, f(t,z_x,u)\rangle - L(t,z_x,u),
    \end{equation}
    where the adjoint variable $p_x$ evolves backward in time and encodes the sensitivity of the optimal cost with respect to the state. An optimal control $u^\star$ must satisfy a two-point boundary value problem consisting of the state dynamics and adjoint dynamics,
    \begin{align}
    \dot z_x(t) &= -\nabla_p \mathcal{H}(t,z_x(t),p_x(t),u^\star(t)),\\
    \dot p_x(t) &= \nabla_z \mathcal{H}(t,z_x(t),p_x(t),u^\star(t)),
    \end{align}
    with $z_x(0)=x$ and $p_x(T)=\nabla G(z_x(T))$,
    along with the optimality condition 
    \begin{equation}
    \begin{split}
    u^\star(t)\in \arg\max_u \mathcal{H}(t,z_x(t),p_x(t),u)
    \\
    \iff
    \nabla_u \mathcal{H}(t,z_x,p_x,u^\star)=0.
    \end{split}
    \label{eq:u_star_def_general}
    \end{equation}
    A classical and particularly effective way to exploit this structure is through its connection to dynamic programming. When the value function is sufficiently regular, the adjoint variable coincides with the gradient of the value function along optimal trajectories, thereby linking the PMP system to the Hamilton--Jacobi--Bellman (HJB) equation. 
    This observation allows for a class of learning-based methods that parameterize the value function and recover feedback controls via Hamiltonian maximization. These types of approaches have demonstrated strong performance in high-dimensional optimal control problems~\cite{onken2022neural, onken2021neural, lin2021alternating}, particularly when the Hamiltonian
    \begin{equation}
    H(t,z_x,p_x)=\sup_u \mathcal{H}(t,z_x,p_x,u)
    \label{eq:straight_H}
    \end{equation}
    and the corresponding optimal controller $u^\star$ admit closed-form expressions~\cite{onken2022neural, onken2021neural, nakamura2021adaptive, zhao2024offline, Verma-HJB-RL}. By embedding the structural relationships implied by PMP directly into the learning process, value function parameterization can offer substantial efficiency gains over methods that directly parameterize the control policy~\cite{onken2022neural, onken2021neural, li2024neural}. However, these advantages \emph{critically depend} on the availability of an explicit Hamiltonian maximizer, which is often not the case in practical control problems.

    \subsection{Our Contribution}

    Recent work by~\cite{gelphman2025end} proposed an end-to-end framework for learning semi-global feedback controllers for~\eqref{eq:originalOC} by directly parameterizing the value function. Their approach embeds the value function within an implicit neural network~\cite{el2021implicit} and leverages Jacobian-Free Backpropagation (JFB)~\cite{fung2022jfb} to enable efficient end-to-end training despite the presence of implicit Hamiltonians. The authors show promising empirical results for control problems of 
moderate dimensions. Their theoretical analysis establishes that JFB 
produces descent directions for \emph{individual trajectories}.
However, this \emph{sample-wise} guarantee does not address the 
\emph{stochastic} optimization setting. Critically, 
sample-wise descent does \emph{not} imply convergence.

    In this work, we extend that framework both theoretically and empirically. Our contributions are summarized as follows:
    \begin{itemize}
        \item \textbf{First convergence analysis for biased SGD 
    in optimal control:} We prove that JFB-based stochastic gradient descent 
    converges to stationary points when training over a distribution of 
    trajectories, despite JFB producing \emph{systematically biased} 
    gradient estimates. This fundamentally differs from~\cite{gelphman2025end}, 
    which only established that individual sample updates move in descent directions. 
        \item \textbf{Theoretical verification in practice:}
    We empirically verify that the assumptions required by our convergence analysis are satisfied in practical training regimes.
        \item \textbf{Scalability to 100-agent problems:} We demonstrate 
    effectiveness on optimal control problems an order of magnitude larger 
    than~\cite{gelphman2025end}, including 100-agent consumption-savings, 
    swarm quadrotor control, and multi-bicycle dynamics where standard 
    implicit differentiation fails due to memory constraints 
    (Section~\ref{sec:experiments}).
    \end{itemize}

    \section{Related Works}
    \label{sec:related_works}
    Recent advances in neural network–based methods have enabled efficient solutions to high-dimensional optimal control problems in settings where the Hamiltonian admits closed-form solutions. Approaches such as Neural-PMP~\cite{gu2022pontryagin}, PMP-Net, and Pontryagin Differentiable Programming~\cite{jin2020pontryagin, jin2021safe} leverage the Pontryagin Maximum Principle (PMP) to construct end-to-end differentiable frameworks for learning optimal controllers. Closely related are value function parameterization methods~\cite{ruthotto2020machine, li2024neural, onken2022neural, onken2021ot, onken2021neural, Verma-HJB-RL, lin2021alternating, vidal2023taming}, which parameterize the value function directly and recover optimal controls via PMP relations. However, these approaches fundamentally rely on the ability to analytically solve the Hamiltonian maximization problem, which restricts their applicability when closed-form solutions are unavailable. 

    To address computational challenges in learning-based control, several differentiable optimization approaches have been proposed. DiffMPC~\cite{amos2018differentiable} differentiates through Model Predictive Control via KKT conditions, while IDOC~\cite{xu2023revisiting} achieves linear-time complexity through direct matrix equation evaluation. Learned MPC methods~\cite{hertneck2018learning} further reduce computational cost using neural approximations, but often at the expense of preserving optimal control structure.

    Most closely related to our work is the recent framework of~\cite{gelphman2025end}, which introduced an implicit value function parameterization approach for optimal control problems with implicit Hamiltonians using implicit neural networks (INNs) and Jacobian-Free Backpropagation (JFB). While that work demonstrated promising empirical behavior and established sample-wise descent properties, it did not address stochastic minibatch training or convergence. This work extends this framework by developing a convergence theory for minibatch JFB and by providing a substantially broader empirical evaluation on high-dimensional optimal control problems.

    \section{Background}
    \label{sec:background}
    \subsection{Optimal Control}
    \label{sec:OC_back}
    
    A standard route to constructing optimal feedback controllers is through the system’s value function, denoted by $\phi(t,z)$. A key result in optimal control theory states that the value function fully characterizes the optimal control policy. In particular, along an optimal trajectory, the adjoint (costate) variable appearing in the Pontryagin Maximum Principle (PMP) coincides with the spatial gradient of the value function. More precisely, Theorem~I.6.2 of~\cite{fleming2006controlled} states
    \begin{equation}
    \label{eq:gradPhi}
    p_x(t) = \nabla_z \phi\big(t, z_x^\star(t)\big).
    \end{equation}
    
    This relationship allows the optimal control defined implicitly by the PMP optimality condition~\eqref{eq:u_star_def_general} to be written explicitly as a feedback law in terms of the value function,
    \begin{equation}
    \label{eq:opt_control_nec}
    u^\star(t)
    =
    u^\star\!\left(
    t,\,
    z_x^\star(t),\,
    \nabla_z \phi\big(t, z_x^\star(t)\big)
    \right).
    \end{equation}
    As a result, learning or approximating the value function immediately yields access to the associated optimal feedback controller.
    
    The key difficulty addressed in this work arises at this stage. Evaluating the feedback law in~\eqref{eq:opt_control_nec} requires solving the Hamiltonian maximization problem in~\eqref{eq:u_star_def_general}. When this maximization admits a closed-form solution, the resulting feedback law can be computed efficiently. However, in many problems of practical interest, no such closed-form expression exists~\cite{betts2010practical, gelphman2025end}. In these cases, computing $u^\star$ becomes computationally challenging, and existing value function–based approaches~\cite{onken2022neural, onken2021neural, meng2025recent, ruthotto2020machine, lin2021alternating} become prohibitively expensive when the control dimension is large.
    
    The value function $\phi$ itself satisfies the Hamilton--Jacobi--Bellman (HJB) partial differential equation~\cite{evans1983introduction},
    \begin{equation}
    \label{eq:HJB}
    -\partial_t \phi(t,z)
    +
    H\!\left(t, z, \nabla_z \phi(t,z)\right)
    =
    0,
    \quad
    \phi(T,z) = G(z),
    \end{equation}
    where the Hamiltonian $H$ is obtained by solving the maximization problem in~\eqref{eq:straight_H}. This PDE formulation provides a complementary, dynamic-programming perspective on optimal control and further highlights the central role of the Hamiltonian maximization in both analysis and computation.

    \subsection{Implicit Deep Learning}
    \label{subsec: implicit_deep_learning}
    
    To address the presence of implicit Hamiltonians in our control framework, we employ an Implicit Neural Network (INN) architecture. INNs define their outputs as fixed points of learnable operators~\cite{el2021implicit}. In our setting, the operator $T_\theta$ is constructed directly from the Hamiltonian optimality condition~\eqref{eq:u_star_def_general}, and its fixed point corresponds to the optimal control. Specifically, the network output $u_\theta^\star$ is defined implicitly through
    \begin{equation}
    u_\theta^\star = T_\theta(u_\theta^\star; t, z),
    \label{eq:fixed_point}
    \end{equation}
    where $\theta \in \mathbb{R}^p$ denotes the network parameters and $(t,z)$ are the input variables.
    
    Unlike standard feedforward architectures, INNs are characterized by implicit fixed-point conditions~\cite{el2021implicit, fung2024generalization}. This modeling paradigm has been successfully applied across a wide range of applications, including optical flow estimation~\cite{bai2022deep}, game-theoretic equilibria~\cite{mckenzie2024three}, maze-solving~\cite{knutson2024logical}, decision-focused learning~\cite{mckenzie2024differentiating}, 
    image classification~\cite{bai2020multiscale}, and inverse problems~\cite{gilton2021deep, Yin2022Learning, liu2022online, heaton2021feasibility, heaton2023explainable}. Since the optimal control $u^\star$ in~\eqref{eq:u_star_def_general} is itself defined implicitly, INNs provide a natural modeling choice; that is, the fixed-point condition directly encodes the notion of optimality.
    
    Training INNs typically requires differentiating through the solution of a fixed-point equation. A standard approach is implicit differentiation~\cite{he2016deep, el2021implicit}, obtained by differentiating both sides of~\eqref{eq:fixed_point},
    \begin{equation*}
    \frac{d u_\theta^\star}{d\theta}(t,z)
    =
    \frac{\partial T_\theta(u_\theta^\star; t,z)}{\partial u}
    \frac{d u_\theta^\star}{d\theta}(t,z)
    +
    \frac{\partial T_\theta(u_\theta^\star; t,z)}{\partial \theta}.
    \end{equation*}
    Rearranging yields
    \begin{equation}
    \frac{d u_\theta^\star}{d\theta}(t,z)
    =
    \mathcal{J}_\theta^{-1}
    \frac{\partial T_\theta(u_\theta^\star; t,z)}{\partial \theta},
    \quad
    \mathcal{J}_\theta
    =
    I -
    \frac{\partial T_\theta(u_\theta^\star; t,z)}{\partial u},
    \label{eq:implicit_gradient}
    \end{equation}
    which requires solving a linear system \emph{for each evaluation of $(t,z)$}.
    
    In the context of feedback control, this cost becomes prohibitive. \emph{The linear system must be solved for every sample, at every time step (of a trajectory), and across all training iterations}. As a result, recent work has focused on reducing the computational burden of training implicit models~\cite{bai2019deep, el2021implicit, fung2022jfb, bolte2024one}. Among these methods, Jacobian-Free Backpropagation (JFB)~\cite{fung2022jfb} offers a particularly simple alternative, closely related to one-step differentiation~\cite{bolte2024one}. JFB replaces the Jacobian inverse $\mathcal{J}_\theta^{-1}$ with the identity, corresponding to a zeroth-order Neumann approximation,
    \begin{equation}
    \frac{d u_\theta^\star}{d\theta}(t,z)
    \approx
    \frac{\partial T_\theta(u_\theta^\star; t,z)}{\partial \theta}.
    \label{eq:jfb_gradient}
    \end{equation}
    
    Despite its simplicity, JFB has been shown to be effective in practice and is straightforward to implement~\cite{fung2022jfb}. An additional advantage of JFB is that it avoids difficulties associated with non-differentiable components in the operator $T_\theta$, such as ReLU activations, making it particularly well suited for the control problems considered in this work.

    \subsection{Training Problem Formulation}
    We formulate training as an optimization over a distribution of initial conditions,
    \begin{subequations}
    \begin{align}
    \min_\theta \; \mathbb{E}_{x \sim \rho} \; J_x(\theta)
    &=
    \int_0^T L(s, z_x, u_\theta^\star)\, ds + G(z_x(T)), \label{eq:training_problem1}\\
    \text{subject to:}\quad
    \dot z_x &= f(t,z_x,u_\theta^\star), \quad z_x(0)=x, \label{eq:training_problem2}\\
    u_\theta^\star &\in \argmax_u \mathcal{H}(t,z_x,\nabla \phi_\theta,u), \label{eq:training_problem3}
    \end{align}
    \end{subequations}
    where $\rho \in \mathbb{P}$ denotes a distribution over initial states defined on the probability space $\langle \Omega, \Sigma, P\rangle$, with $P$ being the probability measure on the space. Solving~\eqref{eq:training_problem1}--\eqref{eq:training_problem3} yields a semi-global value function and, consequently, a feedback controller, since training is performed over a family of initial conditions.
    
    The main computational challenge lies in evaluating and differentiating through the implicitly defined control $u_\theta^\star$. As discussed in Section~\ref{subsec: implicit_deep_learning}, $u_\theta^\star$ is represented as the fixed point of an operator $T_\theta$ derived from the Hamiltonian optimality condition. For example, $T_\theta$ may be instantiated by an algorithmic update whose fixed points satisfy the first-order optimality condition of~\eqref{eq:training_problem3}, such as a gradient-based iteration, a proximal or operator-splitting scheme, or another monotone operator, while remaining differentiable with respect to the parameters $\theta$. However, this implicit representation introduces nontrivial computational costs during training, since differentiation must account for the dependence of the fixed point on $\theta$.
    
    \section{Convergence Analysis}
    \label{sec:convergence}
    In this section, $z$ is interchangeable with $z_x$.
    \subsection{Notation and Essential Assumptions}
    \label{subsec:essential_assumptions}
    \begin{assumption}[Smoothness and Contractivity] The following holds:
        \begin{enumerate}
            \item There exists $\gamma \in  (0,1)$ such that the operator $T_\theta$ is $\gamma$-contractive in $u$ for all $t \in [0,T], z \in \mathbb{R}^n$ and $ \theta \in \mathbb{R}^p$.
            \item $T_{\theta}$ is $C^1$ with respect to $\theta, t, u, z$. Furthermore, the gradient with respect to $\theta$ of the functional
            \begin{equation}
            \label{eq:J}
            J_x[u_{\theta}] = \int_{0}^{T}L(t,u_{\theta},z)dt + G(z(T)) 
            \end{equation}
            is $L_J$-Lipschitz and the functional ~\eqref{eq:J} is bounded from below by a constant $J_{\text{inf}}$ in some open subset of its domain.
        \end{enumerate}  
        \label{assumption:T}
    \end{assumption}
    \begin{definition}[Objective and Gradient and JFB]
        The (true) gradient and JFB gradient approximation for a sample from the objective~\eqref{eq:training_problem1}--\eqref{eq:training_problem3} are defined as
        \begin{equation}
        \nabla_{\theta}J_x = \int_{0}^{T}v_{\theta,x}(t) dt \quad \text{ and } \dJFB = \int_{0}^{T}w_{\theta,x}(t) dt,
        \label{eq:full_derivatives}
        \end{equation}
        respectively, where 
        \begin{equation}
            \begin{split}
            v_{\theta,x}(t) = \frac{d u_\theta^\star}{d \theta}^{\top}h_{\theta, x}, \quad \text{ and } \quad 
            w_{\theta,x}(t) = \frac{\partial T_\theta}{\partial  \theta}^{\top}h_{\theta, x},
            \end{split}
            \label{eq:integrand_definitions}
        \end{equation} 
        and $h_{\theta, x} = (\nabla_u L(t,z_{x}, u_\theta^\star) + \nabla_u f^{\top}p_{x})$ represents the gradient of the Hamiltonian, $p_x$ is the adjoint variable satisfying the adjoint equation~\cite{evans1983introduction}, and
        $\dfrac{du_\theta^\star}{d\theta}$ is the implicit gradient given by~\eqref{eq:implicit_gradient}.
    \end{definition}
    Note that JFB \emph{circumvents the expensive computation of $\frac{du_\theta^\star}{d\theta}$}, and instead replaces it with $\frac{\partial T_\theta}{\partial \theta}$, resulting in significantly reduced computational cost.
  
    For the remainder of this work, $\sigma_{min}(\cdot)$ and $\sigma_{max}(\cdot)$ are functions that, respectively, extract the smallest and largest singular values of its inputs, assumed to be a matrix. 

    \begin{assumption}[Bound on Hamiltonian Gradient]
        There exists $B_{max} > 0$ such that for all $t,z,\theta$, we have 
        \begin{equation}
            \left\| h_{\theta, x}(t,z,u_\theta^\star) \right\| = \|\nabla_uL + \nabla_uf^{\top}p_x\| \leq B_{\max},
        \end{equation}
        where $p_x$ is the adjoint variable satisfying the adjoint equations~\cite{evans1983introduction}.
        \label{assumption:h}
    \end{assumption}
    Assumption~\ref{assumption:h} requires the gradient of the Hamiltonian (with respect to the control) to be uniformly bounded above by $B_{\text{max}}$. 
    \begin{assumption}[Conditioning on JFB Integrand Matrix]
        \label{assumption:M}
        For any $\theta, t, z$, the matrix $M_\theta = \frac{\partial T_{\theta}}{\partial \theta}(u_\theta^\star; t, z) \in \mathbb{R}^{m \times p}$ has full row rank and satisfies the upper bound on the singular values
        \begin{equation}
        \sigma_{\max}(M_\theta) \leq \frac{1}{\sqrt{\beta}},
        \label{eq:bounded_singular_value_M}
        \end{equation}
        where $\beta > 0$. Moreover, the Gram matrix $M_{\theta}M_{\theta}^{\top}$ is nonsingular and $\left( M_\theta M_\theta^\top \right)^{-1}$ satisfies the upper bound on the condition number
        \begin{equation}
        \kappa((M_{\theta}M_{\theta})^{-1}) = \frac{\lambda_{max}((M_{\theta}M_{\theta})^{-1})}{\lambda_{min}((M_{\theta}M_{\theta})^{-1})} < \frac{1}{\gamma}.
        \end{equation}
    \end{assumption}
    Assumption~\ref{assumption:M} is based on the original paper on JFB~\cite{fung2022jfb}, which requires conditioning restrictions on the JFB update $M_\theta$. For brevity, we denote $Var_x[v_{\theta,x}(t)] := \mathbb{E}_x[\|v_{\theta,x}-E_v\|^2]$ and similarly for $Var_x[w_{\theta,x}(t)]:= \mathbb{E}_x[\|w_{\theta,x}-E_w\|^2]$, where $E_v = \mathbb{E}_x[v_{\theta,x}]$ and $E_w = \mathbb{E}_x[w_{\theta,x}]$.
    \begin{assumption}[Variance Bound]
        \label{assumption:expectation_integrand_inner_product}
        $\forall \theta, t, z$, $\exists 0 < \delta_{var} < \lambda_{-} - \gamma \lambda_{+}$ such that 
        \begin{equation}
        \max\left(\sqrt{\text{Var}_x[v_{\theta,x}(t)]}, \sqrt{\text{Var}_x[w_{\theta,x}(t)]}\right)^2 \leq \delta_{var}\|\mathbb{E}_x[M_{\theta}v_{\theta,x}]\|^2,  
        \label{eq:variance_vw}
        \end{equation}
        where $\lambda_{+}$ and $\lambda_{-}$ are \emph{uniform bounds} on the largest and smallest eigenvalues of $(M_\theta M_\theta^\top)^{-1}$ over $\theta, t,z,$ respectively.
    \end{assumption}
        Assumption~\ref{assumption:expectation_integrand_inner_product} is possibly the most important assumption needed to prove convergence in expectation (and probability) to a critical point of~\eqref{eq:training_problem1}--\eqref{eq:training_problem3}. In words, this assumption controls how noisy the sample-wise integrands $v_{\theta,x}$ and $w_{\theta,x}$ are relative to their mean direction. This assumption is similar to variance bounds used in standard proofs of stochastic gradient descent (SGD)~\cite{bottou2018optimization}. 

\begin{assumption}
    Let 
    \begin{equation}
        C_v = \frac{1}{T} \int_0^T v_{\theta,x}(t) dt \quad \text{ and } \quad C_w = \frac{1}{T} \int_0^T w_{\theta,x}(t) dt, 
    \end{equation}
    where $v_{\theta,x}, w_{\theta,x}$ are defined in~\eqref{eq:integrand_definitions}. 
    We have that for all $\theta, x, z$, the following hold.
    \begin{enumerate}
        \item Each element in the vectors $v_{\theta,x}, w_{\theta,x}$ is integrable on $[0,T]$ with respect to $t$. Moreover, $v_{\theta,x},w_{\theta,x}$ are integrable on $[0,T] \times \Omega$, where $\Omega$ is the sample space of the distribution of initial conditions, $\rho$.
        \item $\exists  \ \delta_v, \delta_w, a_v, a_w \geq 0 \text{ and }\epsilon_v > 0$ such that
        \begin{align*}
        &\|\mathbb{E}_x[v_{\theta,x}(t) - C_v]\| \leq a_v + \delta_v \left\|\mathbb{E}_x[ \nabla_{\theta}J_x]\right\|, \\  
        &\|\mathbb{E}_x[w_{\theta,x}(t) - C_w]\| \leq a_w + \delta_w \left\| \mathbb{E}_x[\dJFB]\right\|
        \end{align*}
        and
        \begin{align*}
        &\max(a_v + \delta_v\left\| \mathbb{E}_x[\nabla_{\theta}J_x]\right\| , a_w + \delta_w \left\| \mathbb{E}_x[\dJFB]\right\|)^2 \\
        & \leq \delta_{\theta}^{2} - \frac{\epsilon_v}{T^2} \left\| \mathbb{E}_x[\nabla_{\theta}J_x]\right\|^2, 
        \end{align*}
        where $\delta_{\theta} = \sqrt{\lambda_{-} - \gamma\lambda_{+} - \delta_{var}}\|\mathbb{E}_x[M_{\theta}v_{\theta, x}(t)]\|$
    \end{enumerate}
\label{assumption:expectation_timeaverage}
\end{assumption}
Assumption~\ref{assumption:expectation_timeaverage} controls the magnitude of temporal deviations of the integrands from their time averages, which ensures that these deviations remain uniformly bounded. 
This condition captures a phenomenon specific to optimal control problems (the accumulation of gradient information over time) and does not arise in standard analysis of SGD. \newline \newline

\begin{remark}
    These assumptions are related to those in~\cite{gelphman2025end}, but differ in several fundamental ways. In particular, the assumptions in~\cite{gelphman2025end} are imposed \emph{sample-wise}, which are significantly stronger than those considered here, where all conditions are formulated in expectation with respect to $x$. Moreover,~\cite{gelphman2025end} assumes the existence of uniform lower bounds on $\|h_{\theta,x}\|$ and $\sigma_{\min}(M_{\theta})$, the smallest singular value of $M_{\theta}$. This distinction is crucial as such lower bounds generally fail to hold at stationary points, thereby precluding convergence guarantees under those assumptions. For this reason, we do not impose lower bounds on $\|h_{\theta,x}\|$ or $\sigma_{\min}(M_{\theta})$ in our analysis.
\end{remark}

\subsection{Alignment of JFB and Gradient Lemmas}
\newcommand{\lemmaExpectationIntegrand}[1]{Under Assumptions \ref{assumption:T}-\ref{assumption:expectation_integrand_inner_product}, 
\begin{equation*}
\langle \mathbb{E}_x[v_{\theta,x}(t)],\mathbb{E}_{x}[w_{\theta,x}(t)]\rangle \geq \delta_{\theta}^{2} \geq 0, \quad \forall \theta, t,z.
\end{equation*}
}
\begin{lemma}
    \label{lemma:expectation_integrand}
    \lemmaExpectationIntegrand{main}
\end{lemma}

\newcommand{\lemmaExpectationIntegralProduct}[1]{
Under Assumptions ~\ref{assumption:T} - ~\ref{assumption:expectation_timeaverage},
\begin{equation*}
\mathbb{E}_x[\nabla_{\theta}J_x]^\top \mathbb{E}_x\left[\dJFB\right] \geq \epsilon_v  \left\| \mathbb{E}_x[\nabla_{\theta}J_x] \right\|^2, \quad \forall z, u, \theta.
\end{equation*}
}

\begin{lemma}
\label{lemma:expectation_integral_product}
\lemmaExpectationIntegralProduct{main}
\end{lemma}

The Lemmas above are particularly important because they show that the JFB update remains positively aligned with the true gradient in expectation, both pointwise in time and after time integration, which leads to descent of the expected objective. Proofs of both Lemmas can be found in the Section~\ref{appendix:main_draft_proof} of the Appendix.

\subsection{Convergence of SGD Using JFB as the Stochastic Gradient}
\label{subsec:convergence}

We begin by introducing notation to formalize the stochasticity arising in the training process. Let $\{\xi_j\}_{j\ge0}$ denote a sequence of independent random variables representing the sampling procedure used to construct the JFB-based stochastic gradient. In particular, $\xi_j$ corresponds to the random draw of initial conditions $x \sim \rho$ used to compute the JFB update at iteration $j$.

We analyze the convergence of SGD when the Jacobian-Free Backpropagation (JFB) direction is used as a stochastic gradient surrogate. Specifically, we consider the iterative scheme
\begin{equation}
\theta_{j+1} = \theta_j - \alpha_j d_{\xi_j}^{\mathrm{JFB}}(\theta_j), \qquad j \ge 0,
\label{eq:SGD_iteration}
\end{equation}
for minimizing the objective in~\eqref{eq:training_problem1}--\eqref{eq:training_problem3} over $\theta \in \mathbb{R}^p$. Here, $d_{\xi_j}^{\mathrm{JFB}}(\theta_j)$ denotes the JFB update computed using either a single sample or a minibatch of samples with corresponding learning rate $\alpha_j$ at iteration $j$.

Following the notation of~\cite{bottou2018optimization}, we use $\mathbb{E}_{\xi_j}[\cdot]$ to denote the conditional expectation with respect to the randomness at iteration $j$, given the current iterate $\theta_j$. Since $\theta_j$ depends on the sequence of random variables $\{\xi_0, \xi_1, \ldots, \xi_{j-1}\}$, we also consider the total expectation of the objective with respect to all prior randomness, which we write as
\begin{equation}
\mathbb{E}[\mathbb{E}_x[J_x(\theta_j)]]
=
\mathbb{E}_{\xi_0}\Big[
\mathbb{E}_{\xi_1}\big[
\cdots
\mathbb{E}_{\xi_{j-1}}\big[
\mathbb{E}_x[J_x(\theta_j)]
\big]
\cdots
\big]
\Big].
\end{equation}

With this notation in place, we establish the following Lemma, which is used to prove the main result.

\newcommand{\lemmaExpectationDescent}[1]{
Under Assumptions ~\ref{assumption:T} - ~\ref{assumption:expectation_timeaverage}, JFB-based SGD iterations ~\eqref{eq:SGD_iteration} satisfy
\begin{equation*}
\begin{split}
\mathbb{E}_{\xi_j}[\mathbb{E}_x[J_x&(\theta_{j+1})]] - \mathbb{E}_x[J_x(\theta_j)] \leq \\
& -\alpha_j \epsilon_v \|\mathbb{E}_x[\nabla_{\theta}J_x(\theta_j)]\|^2 + \frac{\alpha_j^2L_JB_{max}^2T^2}{2 \beta}
\end{split}
\end{equation*}}

\begin{lemma}
\label{lemma:expectation_descent}
\lemmaExpectationDescent{main}
\end{lemma}

With this result, the main results of this paper can be proven.

\newcommand{\theoremExpectationCesaroConvergence}[1]{
Suppose the sequence of learning rates $\{\alpha_j\}_{j=0}^{\infty}$ is monotonically decreasing and satisfies $\sum_{j=0}^{\infty}\alpha_j = \infty$, $\sum_{j=0}^{\infty}\alpha_{j}^2 < \infty$, and $0 < \alpha_0 \leq \frac{2 \epsilon_v}{L_{J}(1-\gamma)^2}$. Let $A_K= \sum_{j=0}^K\alpha_j$. Then, under Assumptions ~\ref{assumption:T}-~\ref{assumption:expectation_timeaverage}, the JFB-based SGD iteration \eqref{eq:SGD_iteration} satisfies
\begin{equation*}
\lim_{K \rightarrow \infty} \mathbb{E}\left[ \frac{1}{A_K}\sum_{j=0}^{K}\alpha_j\left\| \mathbb{E}_x[\nabla_{\theta}J_{x}(\theta_j)]\right\|^2 \right] = 0.
\end{equation*}}

\begin{theorem}
\label{theorem:expectation_cesaro_convergence}
\theoremExpectationCesaroConvergence{main}
\end{theorem}
In words, the weighted Cesaro sum of the sequence $\{ \left\| \mathbb{E}_{x}[\nabla_{\theta}J_{x}(\theta_j)]\right\|^2\}_{j=0}^{\infty}$ converges in (total) expectation to 0. 
Using Theorem~\ref{theorem:expectation_cesaro_convergence}, one can then use standard SGD analysis to show the following theorem and corollary.

\newcommand{\theoremExpectationLiminf}[1]{
Under the assumptions of Theorem ~\ref{theorem:expectation_cesaro_convergence}, the JFB-based SGD iteration \eqref{eq:SGD_iteration} satisfies
\begin{equation*}
\liminf_{j \rightarrow \infty} \mathbb{E}\left[\left\| \mathbb{E}_{x}[\nabla_{\theta}J_{x}(\theta_j)]\right\|^2 \right] = 0.
\end{equation*}
}

\begin{theorem}
\label{theorem:expectation_liminf}
\theoremExpectationLiminf{main}
\end{theorem}

Using Theorem ~\ref{theorem:expectation_cesaro_convergence}, we can also prove convergence in probability to a critical point.
\newcommand{\corollaryConvergenceProbability}[1]{
Suppose the assumptions of Theorem ~\ref{theorem:expectation_cesaro_convergence} hold. For any $K \in \mathbb{N}$ let $j(K) \in \{0,1,...,K\}$ represent a random index chosen with probabilities proportional to $\{\alpha_j\}_{j=0}^{K}$. Then, $\{\left\| \mathbb{E}_x[\nabla_{\theta}J_{x}(\theta_j)]\right\|\}_{j=0}^{K} \rightarrow 0$ as $K \rightarrow \infty$ in probability.}
\begin{corollary}
\label{corollary:convergence_probability}
\corollaryConvergenceProbability{main}
\end{corollary}

Finally, we note that it is possible to relax Assumption~\ref{assumption:expectation_integrand_inner_product} and instead obtain convergence to a \emph{neighborhood} of the stationary point instead. See Assumption ~\ref{assumption:app_weak_bound_on_integrand_var} and Theorem ~\ref{theorem:appendix_theorem} in Section C of the Appendix for more details on this weaker assumption. Moreover, under an assumption that is slightly stronger than Assumption ~\ref{assumption:app_weak_bound_on_integrand_var} but significantly weaker than Assumption ~\ref{assumption:expectation_integrand_inner_product}, to prove convergence in expectation to a critical point. See Assumption ~\ref{assumption:limiting_bound_on_integrand_var} and Theorem ~\ref{theorem:limiting_theorem} in Section D of the Appendix for more information.

\section{Experiments}
\label{sec:experiments}
We first empirically verify the assumptions underlying our convergence analysis that can be evaluated directly in practice. We then compare JFB with existing backpropagation methods for implicit neural networks on low- to moderate-dimensional benchmarks. Finally, we demonstrate the effectiveness and scalability of the proposed approach on a range of high-dimensional optimal control problems, including optimal multi-agent consumption–savings and multi-agent quadrotor and bicycle control. In the quadrotor and consumption–savings settings, the use of an exponential running cost leads to an implicit $H$ (see Section~\ref{app:consumption_description} for a description of consumption savings Hamiltonian), while in the multi-agent bicycle problem, the nonlinearity of the dynamics automatically renders $H$ implicit. While the quadrotor and multi-bicycle dynamics are well established, the optimal consumption–savings model is less standard; we therefore provide a detailed description of this problem in Appendix~\ref{app:consumption_description}.

For concreteness, we instantiate the fixed-point operator $T_\theta$ using a simple gradient-based update derived from the Hamiltonian optimality condition; however, our theoretical results apply more broadly to general choices of $T_\theta$ that satisfy the stated assumptions.

\begin{figure}[t]
        \centering
        \includegraphics[width=\linewidth]{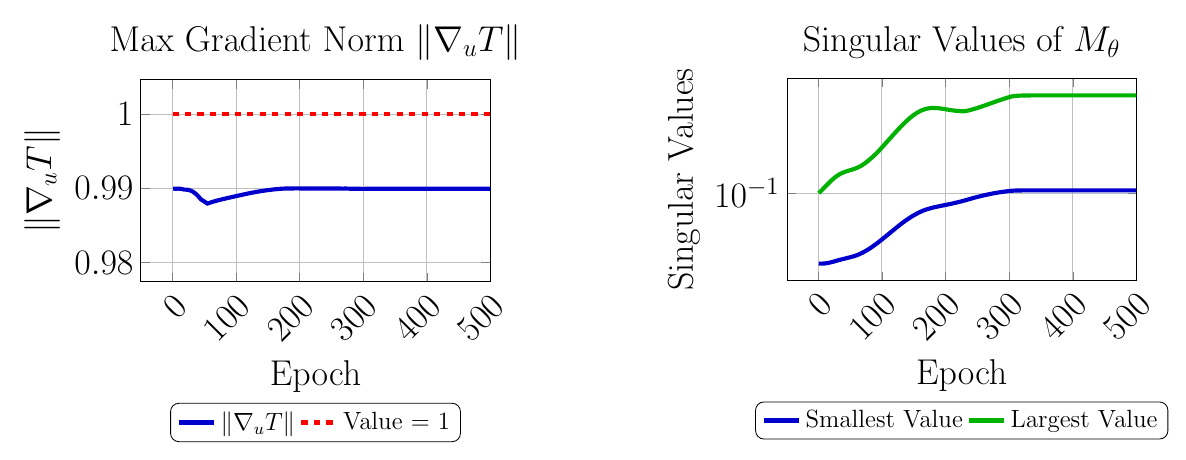}
        \caption{\small{Numerical verification on quadrotor experiment. (Left) Maximum singular value of the Jacobian of the operator $T_\theta$, evaluated on the worst-case sample within each minibatch, illustrating that the operator remains contractive throughout training. (Right) Smallest and largest singular values of $M_{\theta}$, computed batch-wise over all time steps for the single quadrotor experiment, confirming the full row-rank property of $M_{\theta}$ and the boundedness of $\sigma_{\max}(M_{\theta})$.}}

        \label{fig:singular_eigen_values}
\end{figure} 

\begin{figure}[t]
        \centering
        {Angle Between $\mathbb{E}_x[\nabla_{\theta}J_x(\theta_j)]$ and $\mathbb{E}_x[\dJFB(\theta_j)]$} 
        \\
        \includegraphics[width=1.0\linewidth]{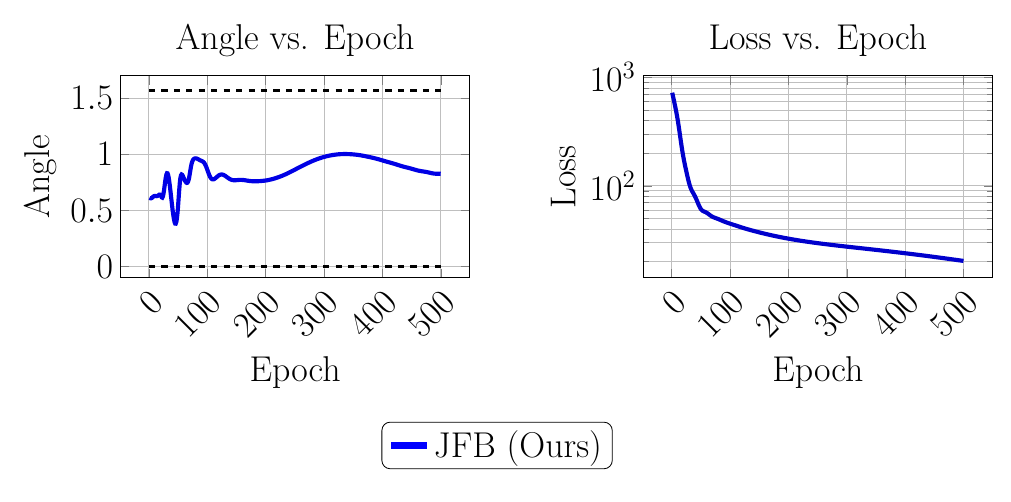}
        \caption{\small{Angle between $\mathbb{E}_x[\nabla_{\theta}J_x]$ and $\mathbb{E}_x[\dJFB]$ (Left) plotted alongside loss vs. epoch (Right). The values plotted are the largest batch-wise angles. $\mathbb{E}_x[\nabla_{\theta}J_x]$ is computed using AD. The dashed lines are at the angles $0$ and $\frac{\pi}{2}$}.}
        \label{fig:angle}
\end{figure}

\subsection{Theoretical Verification}
To empirically verify key theoretical properties of JFB, we consider an optimal control task involving a single quadrotor. We focus on assumptions from the convergence analysis that can be evaluated directly during training, without requiring access to the true optimal solution or adjoint variable. 

\textbf{Contractivity}. We examine whether the operator $T_\theta$, instantiated via a gradient-based update with fixed stepsize 0.01 to solve~\eqref{eq:training_problem3}, is contractive. Since the contractivity factor depends on the learned parameters $\theta$, this property cannot be directly enforced and must instead be verified empirically. As shown in Figure~\ref{fig:singular_eigen_values}, the operator remains contractive at every training iteration, where the reported gradient corresponds to the worst-case sample within each minibatch. Notably, the network consistently learns parameters that preserve contractivity throughout training.

\textbf{Rank and Conditioning of $M_\theta$}. Next, we examine the full-rank condition and boundedness of the largest singular value of $M_\theta$ appearing in Assumption~\ref{assumption:M}. As shown in Figure~\ref{fig:singular_eigen_values}, the singular values of $M_\theta$ remain bounded above and bounded away from zero throughout training, which ensures that $M_\theta$ retains full rank. Consequently, $\lambda_{\max}\!\left((M_\theta M_\theta^\top)^{-1}\right)$ remains bounded above and $\lambda_{\min}\!\left((M_\theta M_\theta^\top)^{-1}\right)$ is bounded away from zero.

\textbf{Descent Verification}. We verify that JFB produces a descent direction in practice. For each minibatch, we compute both the average JFB update and the corresponding average of the true gradients, and evaluate the angle between them. As shown in Figure~\ref{fig:angle} (left panel), this angle remains strictly between $0$ and $\pi/2$ throughout training, indicating consistent alignment. In particular, the angle between $\mathbb{E}_x[\nabla_{\theta}J_x]$ and $\mathbb{E}_x[\dJFB]$ remains well within $[0,\pi/2)$, which implies that the inner product $\mathbb{E}_x[\nabla_{\theta}J_x]^\top \mathbb{E}_x[\dJFB]$ is uniformly bounded away from zero.

While it is not practical to verify each remaining assumption leading to Lemma~\ref{lemma:expectation_descent} individually, we instead demonstrate a consistent decrease in the training loss (right panel of Figure~\ref{fig:angle}). This observed descent provides empirical evidence that the remaining assumptions are reasonable in practice and that the conclusion of Lemma~\ref{lemma:expectation_descent} holds. Accordingly, the angle and descent checks presented here constitute the most meaningful empirical validation of the theory, as they directly assess whether the stochastic JFB updates act as descent directions. Moreover, it was not possible to perform the same numerical verification on larger experiments due to the excessive memory consumption of automatic differentiation needed to compute the exact gradient.
\begin{figure}[t]
        \centering
        \textbf{Single Quadrotor} 
        \\
        \includegraphics[width=1.0\linewidth]{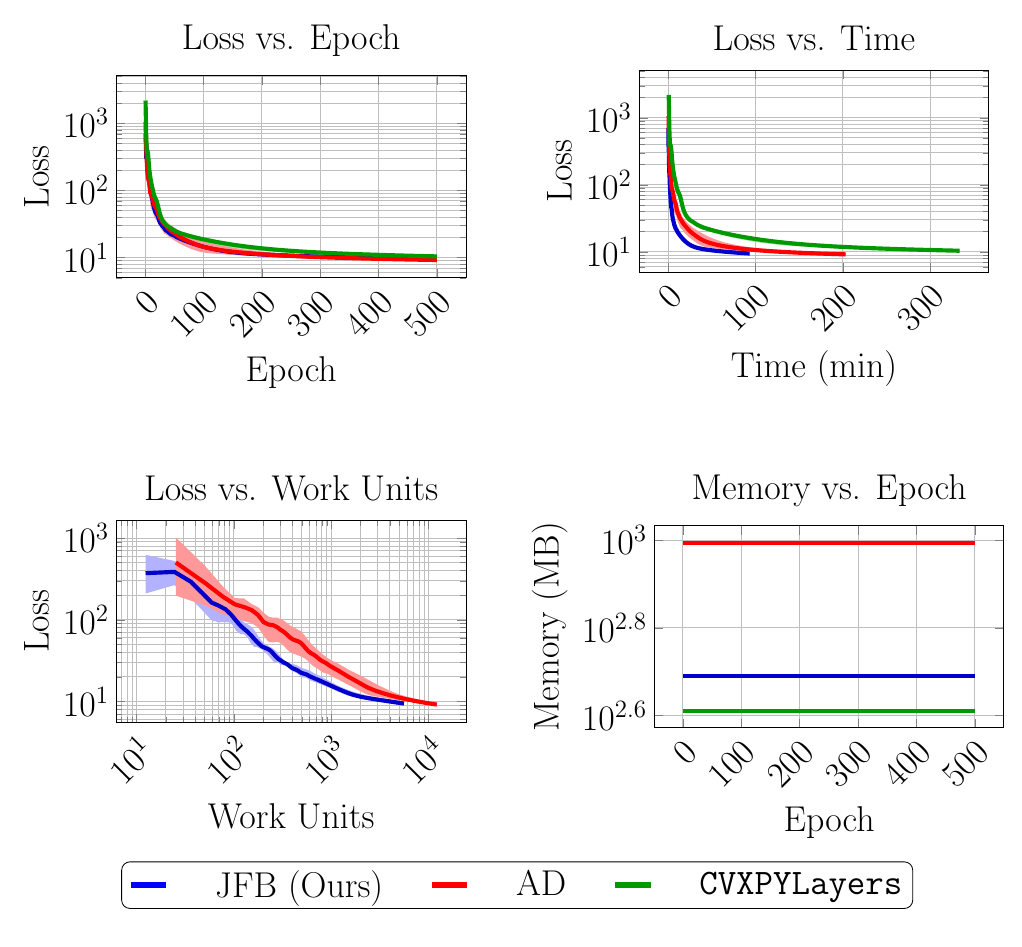}
        \caption{\small{Comparison of JFB, automatic differentiation (AD), and \texttt{CVXPYLayers}~\cite{agrawal2019differentiable} (Implicit Differentiation) for training the value function for a quadrotor across four metrics. (Top Left) Loss versus training epochs. (Top Right) Loss plotted against cumulative runtime in minutes. (Bottom Left) Loss plotted against cumulative work units, with one work unit being one evaluation of $\frac{\partial T_{\theta}}{\partial \theta}$, which is equivalent to backpropagation through one application of $T_\theta$. (Bottom Right) Maximum GPU memory usage per training epoch.}}
        \label{fig:single_quadrotor}
\end{figure}
\begin{figure}[t]
        \centering
        \textbf{6 Quadrotors} 
        \\
        \includegraphics[width=1.0\linewidth]{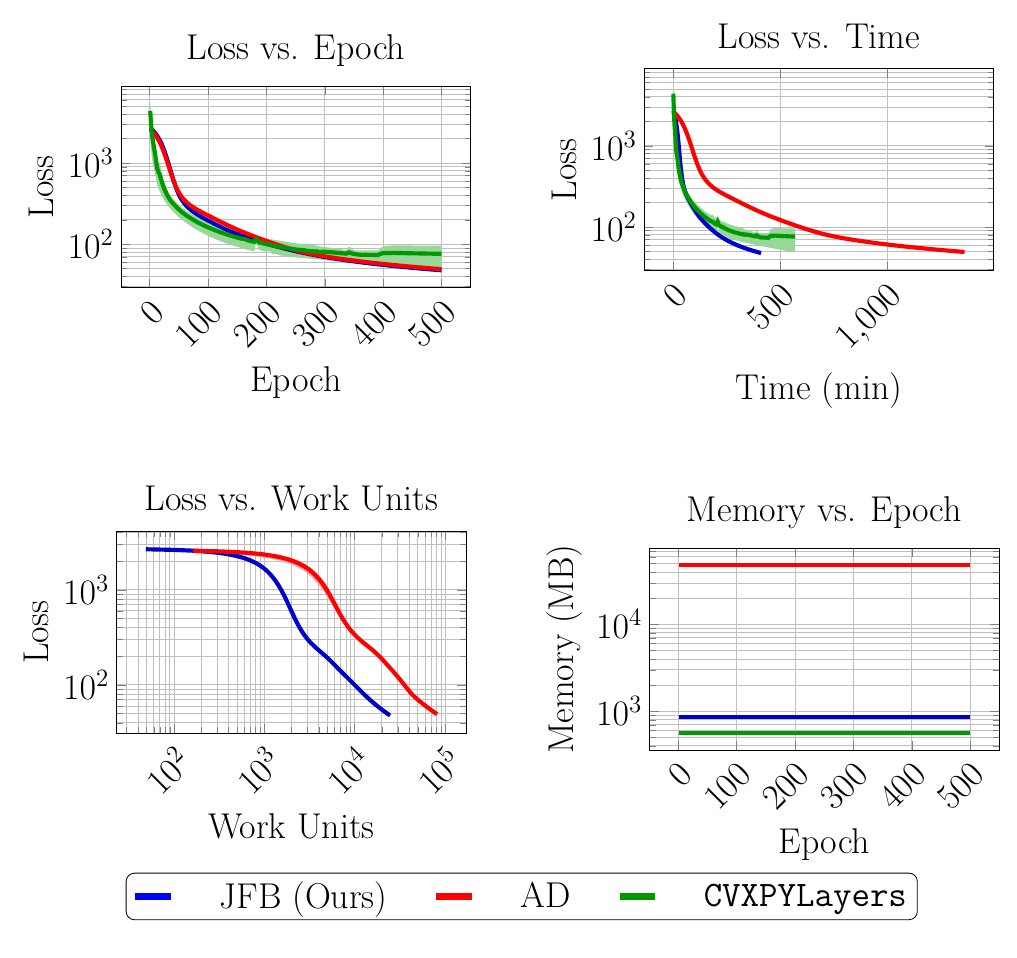}
        \caption{\small{Comparison of JFB, automatic differentiation (AD), and \texttt{CVXPYLayers} (Implicit Differentiation) for training the value function for 6 quadrotors across three metrics. (Top Left) Loss versus training epochs. (Top Right) Loss plotted against cumulative runtime in minutes. (Bottom Left) Loss plotted against cumulative work units. (Bottom Right) Maximum GPU memory usage per training epoch.}}
        \label{fig:6_quadrotors}
\end{figure}
\begin{figure}[t]
    \centering
    \begin{tabular}{ccc}
        \small{\textbf{JFB}}
        &
        \small{\textbf{CVXPYLayers}}
        &
        \small{\textbf{AD}}
        \\        \includegraphics[width=0.3\linewidth]{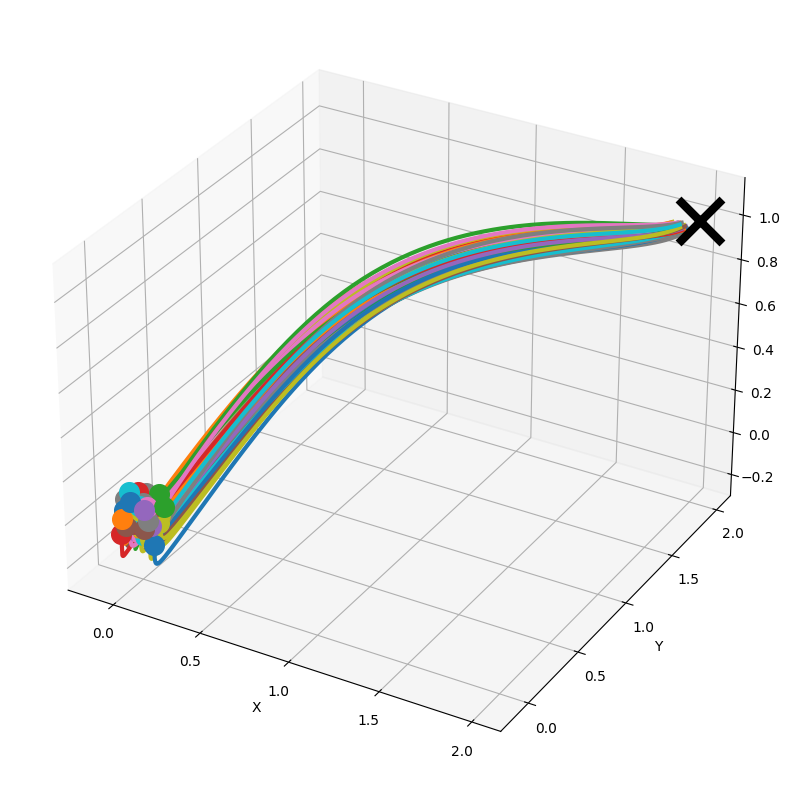}
        &
        \includegraphics[width=0.3\linewidth]{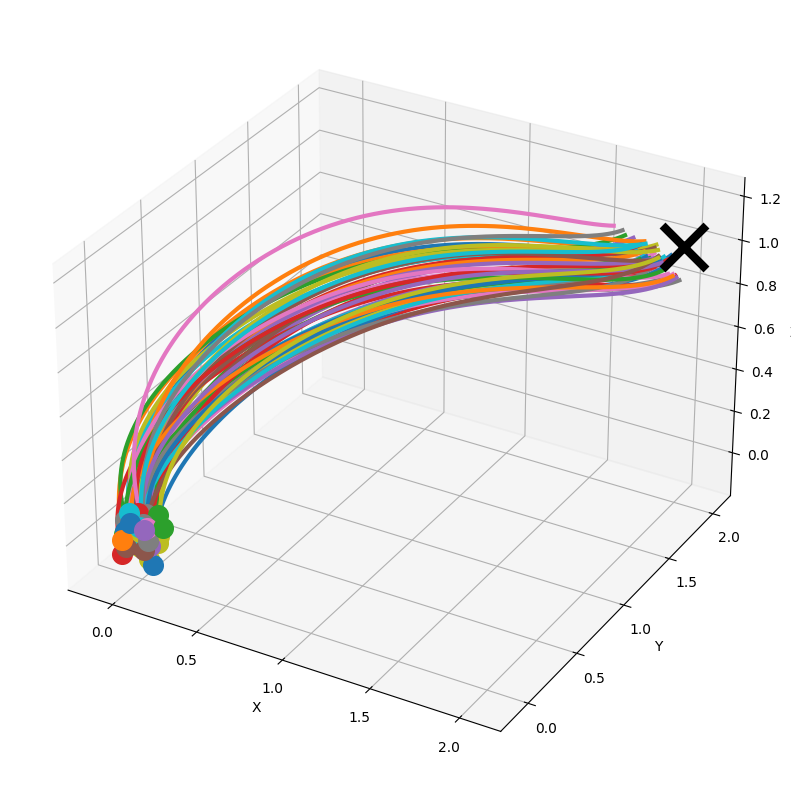}
        &
        \includegraphics[width=0.3\linewidth]{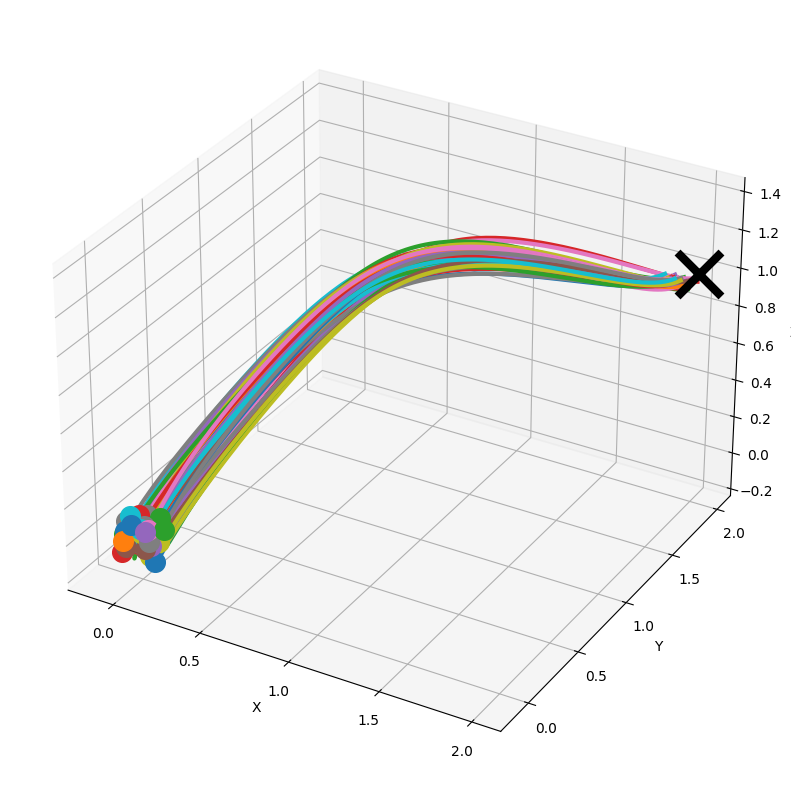}
    \end{tabular}
    \caption{Trajectories of trained quadrotors using JFB (left panel), \texttt{CVXPYLayers} (middle panel), and AD (right panel).}
    \label{fig:quadrotor_trajectories}
\end{figure}

\subsection{Comparison in Low- and Moderate-Dimensional Settings}

We compare several implicit deep learning approaches, including the proposed JFB method, \texttt{CVXPYLayers}~\cite{agrawal2019differentiable}, and traditional automatic differentiation (AD)~\cite{paszke2019pytorch}. To ensure a fair comparison, we focus on low- and moderate-dimensional problems, specifically, single-quadrotor and six-quadrotor control tasks, which remain tractable for all methods.
For all experiments, learning rates were chosen to yield the best empirical performance. Specifically, PyTorch's \texttt{ReduceLROnPlateau} LR scheduler was used with initial learning rate $0.01$, decay factor $0.5$, and patience $10$.

\begin{table}[t]
  \centering
  \sisetup{
    detect-weight=true,
    detect-inline-weight=math,
    table-number-alignment=center,
    table-format=4.3
  }
  \captionof{table}{Final Control Objective ~\eqref{eq:training_problem1}-\eqref{eq:training_problem3} Comparison for Low- and Moderate-dimensional Problems.}
  \label{tab:final_losses}
  \small
  \setlength{\tabcolsep}{6pt}
  \begin{tabular}{l
                  S[table-format=4.3]
                  S[table-format=4.3]
                  S[table-format=4.3]}
    \toprule
    & {AD} & \texttt{CVXPYLayers} & JFB \\
    \midrule
    Single Quadrotor               & 37.720 & 93.653   & \textbf{30.406} \\
    Six Quadrotors                & 72.846 & 39.458 & \textbf{24.150} \\
    \bottomrule
  \end{tabular}
\end{table}

As shown in Figures~\ref{fig:single_quadrotor} and~\ref{fig:6_quadrotors}, all methods exhibit comparable loss trajectories when measured against training epochs (top left panels). However, JFB converges substantially faster when performance is evaluated against wall-clock time (top right panels). This efficiency arises from avoiding both the storage of fixed-point iterations required by AD (bottom right) and the per-time-step linear solves required by \texttt{CVXPYLayers}. In addition, we visualize the resulting position trajectories for the quadrotor problem in Figure~\ref{fig:quadrotor_trajectories}, which shows that JFB produces trajectories that are qualitatively comparable to those obtained using automatic differentiation and noticeably closer to the target than those obtained using \texttt{CVXPYLayers}.

To provide a fairer comparison between AD and JFB, we also report loss versus \emph{work units}, defined as the number of applications of $\partial T_\theta / \partial \theta$ (bottom left panels). We do not report this metric for \texttt{CVXPYLayers}, as its gradient computations are performed entirely as a black box; however, the reported runtimes indicate that it is likely less efficient than JFB but more efficient than AD. Finally, we note that the six-quadrotor setting represents the highest-dimensional problem that could be solved using AD before encountering memory limitations.
Finally, Table~\ref{tab:final_losses} reports the final control objective values after training for both the single- and six-quadrotor experiments. While all methods achieve comparable final losses, JFB reaches these solutions with substantially faster runtimes and fewer work units, which highlights its computational efficiency.

\begin{figure}[t]
\centering
    \textbf{Control Objectives}
    \\
  \includegraphics[width=\linewidth, trim=28 0 0 0, clip]{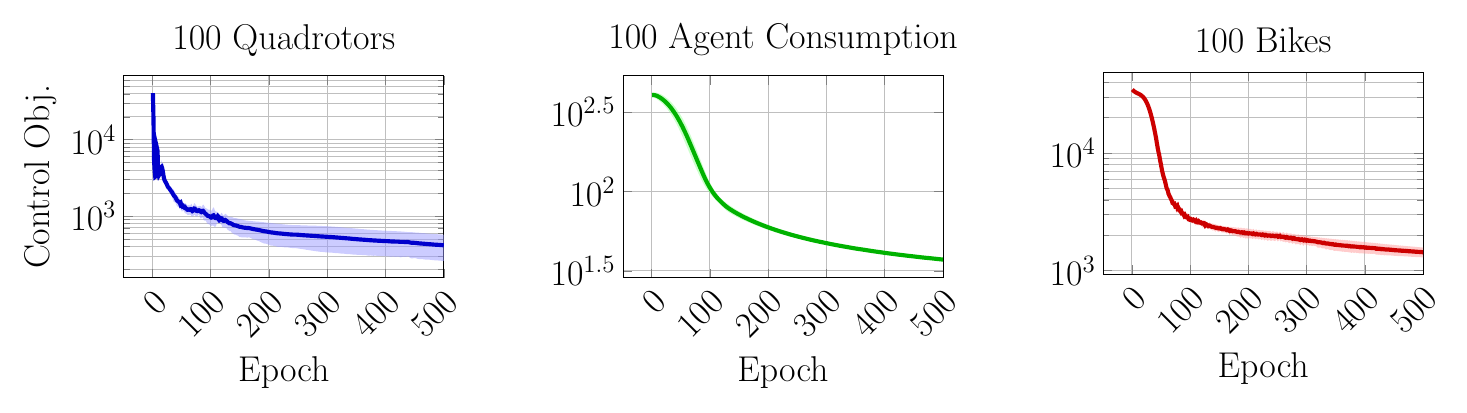}
  \caption{Training trajectories of the control objective~\eqref{eq:training_problem1}-\eqref{eq:training_problem3} for the high-dimensional problems using JFB. Solid lines show the mean over three runs, with shaded regions indicating the best and worst values per epoch. The memory required per epoch is as follows. For both 100 quadrotors and 100 bicycles, JFB uses $36287.0$ MB vRAM per epoch. For 100 agent consumption savings, JFB uses $5411.0$ MB vRAM per epoch.}
  \label{fig:jfb_high_dimensional_problems}
\end{figure}

\subsection{High-Dimensional Experiments}
We next highlight the effectiveness of JFB on high-dimensional optimal control problems with implicit Hamiltonians. Traditional automatic differentiation (AD) is \emph{infeasible} in these settings due to excessive memory consumption, as it requires storing every application of $T_\theta$ within the computational graph (see~\cite{fung2022jfb} for a detailed discussion). \texttt{CVXPYLayers} is similarly \emph{infeasible} for the consumption–savings and bicycle problems, since neither admits dynamics that are linear in the control variable $u$, and is prohibitively slow for the 100-quadrotor experiment. In contrast, JFB enables us to train these networks and approximate solutions for all three high-dimensional problems, as seen in Figure~\ref{fig:jfb_high_dimensional_problems} while also maintaining \emph{constant} memory per epoch. Here, the 100 quadrotor problem is 1200-dimensional in the state and 400-dimensional in the controller. The 100 agent consumption problem is 100-dimensional in both state and controller. The 100 bikes are 400-dimensional in the state and 200-dimensional in the controller.

All high-dimensional experiments employed PyTorch’s \texttt{ReduceLROnPlateau} learning rate scheduler with a decay factor of $0.5$ and patience of $10$. Additional hyperparameters, such as batch size, were tuned individually for each experiment to achieve the lowest control objective.

\section{Discussion}
This work provides the first convergence theory for training implicit neural networks in optimal control problems with implicit Hamiltonians using a biased stochastic gradient method, namely Jacobian-Free Backpropagation (JFB). Our analysis shows that, despite the bias introduced by ignoring Jacobian inversion \emph{at each time step, and for each trajectory} in the gradient estimate, we can ensure convergence under standard smoothness, contractivity, and variance assumptions. Empirically, we demonstrate that key theoretical conditions, such as contractivity, conditioning of ${M}_\theta$, and descent, can be monitored during training, and that JFB scales effectively to high-dimensional control problems where AD and optimization-based methods are infeasible.

\textbf{Limitations}. The primary limitation of this work is its reliance on the contractivity of the operator $T_\theta$. Although we empirically observe this property to hold in our experiments, contractivity may fail for more sophisticated optimization schemes used to evaluate~\eqref{eq:training_problem3}, as well as in settings where the Hamiltonian $\mathcal{H}$ is only locally, rather than strongly, concave. Consequently, an important direction for future work is to extend the convergence analysis to non-contractive or averaged operators. In addition, relaxing the analysis to accommodate cases in which $u_\theta^\star$ corresponds to a local maximizer of $\mathcal{H}$ would further broaden the applicability of the framework.

\section{Conclusion}
\label{sec:conclusion}
We develop the first convergence guarantees for Jacobian-Free Backpropagation (JFB) as a biased stochastic gradient method in optimal control problems with implicit Hamiltonians. Our results show that JFB converges to stationary points despite the absence of exact gradient information and enable scalable learning in high-dimensional control settings. Together, the theory and experiments establish JFB as a principled and practical approach for learning value function–based feedback controllers when closed-form Hamiltonians are unavailable. Future work will extend this framework to mean-field control/games settings~\cite{lasry2007mean, vidal2025kernel, agrawal2022random, lauriere2022learning, wang2025primal, chow2022numerical, yang2023relative, yang2026pathwise}, where the scalability of our approach may prove useful. Code for this work will be provided upon publication.

\section{Acknowledgements}
SWF and EG were partially funded by NSF Award 2309810. SO was partially funded by DARPA under grant HR00112590074, NSF under grant 2208272 and 1554564,
AFOSR under MURI grant N00014-20-1-278, and by ARO
under grant W911NF-24-1-015. SWF and EG were partially funded by NSF Award
2309810. SO was partially funded by DARPA under grant
HR00112590074, NSF under grant 2208272 and 1554564,
AFOSR under MURI grant N00014-20-1-278, and by ARO
under grant W911NF-24-1-015.

\bibliographystyle{icml2026}
\bibliography{references}

\newpage
\appendix
\onecolumn
\section{Proof of Main Result}
\label{appendix:main_draft_proof}
In this section, we provide proofs to all the Lemmas and Theorems. For ease of presentation, we re-state the Lemmas and Theorems before proving them.

\subsection{Proof of Lemma ~\ref{lemma:expectation_integrand}}
\textit{Lemma ~\ref{lemma:expectation_integrand}}:
\lemmaExpectationIntegrand{app}
\begin{proof}
\textit{Outline of Proof:} The proof is carried out in four main steps.
\begin{itemize}
    \item[] \textbf{Step 1.} Fixing $x$, reformulate $\langle v_{\theta,x}, w_{\theta,x}\rangle$ in terms of $M_{\theta} v_{\theta,x}$.
    \item[] \textbf{Step 2.} Use this new formulation and the given assumptions to derive the single-sample (no expectation with respect to $x$) form of desired inequality.
    \item[] \textbf{Step 3.} Take expectation with respect to $x$ on both sides of the result of Step 2 and apply Jensen's inequality.
    \item[] \textbf{Step 4.} Using Assumption ~\ref{assumption:expectation_integrand_inner_product}, algebraically rearrange the result of Step 3 to obtain the desired result.
\end{itemize}
\textit{Step 1}: 
Let $\psi \coloneqq M_{\theta}v_{\theta,x}$. Substituting the definition of $v_{\theta,x}$:
\begin{small}
\begin{equation}
\psi = M_{\theta}(M_{\theta}^\top \Jcal_{\theta}\invT h_{\theta, x}) = (M_\theta M_\theta^\top)\Jcal_{\theta}\invT h_{\theta, x}.
\end{equation}
\end{small}
By Assumption ~\ref{assumption:M}, $M_{\theta}M_{\theta}^{\top}$ is nonsingular. The matrix $\mathcal{J}_\theta$, and therefore $\mathcal{J}_\theta^{-\top}$, is assumed to be nonsingular. 

Our goal is to express $\inner{v_{\theta,x}}{w_{\theta,x}}$ in terms of $\psi$. Given $\psi = (M_\theta M_\theta^\top)\Jcal_\theta\invT h_{\theta, x}$, we have
\begin{small}
\begin{equation}
    h_{\theta, x} = \Jcal_{\theta}\T(M_\theta M_\theta^\top)^{-1}\psi.
\end{equation}
\end{small}
Substituting this back into the definitions, we obtain
\begin{small}
\begin{align*}
v_{\theta,x} &= M_{\theta}^\top \Jcal_{\theta}\invT[\Jcal_{\theta}\T(M_\theta M_\theta^\top)^{-1}\psi] = M_\theta\T(M_\theta M_\theta^\top)^{-1}\psi \\
w_{\theta,x} &= M\T[\Jcal_{\theta}\T(M_\theta M_\theta^\top)^{-1}\psi] = M\T\Jcal_\theta\T(M_\theta M_\theta^\top)^{-1}\psi
\end{align*}
\end{small}
Therefore, $\inner{v_{\theta,x}}{w_{\theta,x}} = \inner{\psi}{\Jcal_\theta \T(M_\theta M_\theta^\top)^{-1}\psi}$, using the definition of adjoint.

\textit{Step 2}: Let $A = (M_\theta M_\theta^\top)^{-1}$. Since $M_\theta M_\theta^\top$ is symmetric positive definite, so is $A$. Let $\lambda_{+}$ and $\lambda_{-}$ denote the largest and smallest eigenvalues of $A$, and define $\bar{\lambda} = \frac{1}{2}(\lambda_{+} + \lambda_{-})$.

Using the assumptions of this Lemma, we have by Lemma A-1 in the original JFB paper (Fung et. al. 2022) ,  $\Jcal_\theta\T$ is coercive , that is,  $\inner{\psi}{\Jcal_\theta\T\psi} \geq (1-\gamma)\norm{\psi}^2$. Using this and Cauchy-Schwarz,
\begin{small}
\begin{align*}
\inner{\psi}{\Jcal_\theta\T A\psi} &= \inner{\psi}{\Jcal_\theta\T(\bar{\lambda}I + A - \bar{\lambda}I)\psi} \\ &=\bar{\lambda}\inner{\psi}{\Jcal_\theta\T\psi} + \inner{\psi}{\Jcal_\theta\T(A - \bar{\lambda}I)\psi} \\
&\geq \bar{\lambda}(1-\gamma)\norm{\psi}^2 - \norm{\Jcal_\theta\T}\norm{A - \bar{\lambda}I}\norm{\psi}^2.
\end{align*}
\end{small}
Using Lemmas A-1 and A-2 from Wu Fung et. al. 2022, we have $\norm{\Jcal_\theta\T} \leq 1+\gamma$ and $\norm{A - \bar{\lambda}I} = \frac{1}{2}(\lambda_{+} - \lambda_{-})$. Substituting these gives,
\begin{small}
\begin{align*}
\inner{v_{\theta,x}}{w_{\theta,x}} \geq \left[ \frac{\lambda_{+} + \lambda_{-}}{2}(1-\gamma) - (1+\gamma)\frac{\lambda_{+} - \lambda_{-}}{2} \right] \norm{\psi}^2 .
\end{align*}
\end{small}
By Assumption ~\ref{assumption:M} the condition number $\kappa(A) < \frac{1}{\gamma} \implies (\lambda_{-} - \gamma\lambda_{+}) > 0$. Thus, we have
$\inner{v_{\theta,x}}{w_{\theta,x}} \geq \norm{\psi}^2(\lambda_{-} - \gamma\lambda_{+}) \geq 0.$

\textit{Step 3}: Taking the expectation with respect to $x$ on both sides of the result of Step 2 and using $\psi=M_{\theta}v_{\theta,x}$ we have
\begin{align*}
\mathbb{E}_x[\inner{v_{\theta,x}}{w_{\theta,x}}] &\geq \mathbb{E}_x[(\lambda_{-} - \gamma\lambda_{+})\norm{M_{\theta}v_{\theta,x}}^2] \\
\mathbb{E}_x[\inner{v_{\theta,x}}{w_{\theta,x}}] &\geq(\lambda_{-} - \gamma\lambda_{+}) \mathbb{E}_x[\norm{M_{\theta}v_{\theta,x}}^2] \\
\mathbb{E}_x[\inner{v_{\theta,x}}{w_{\theta,x}}] &\geq (\lambda_{-} - \gamma\lambda_{+})\norm{\mathbb{E}_x[M_{\theta}v_{\theta,x}]}^2 
\end{align*}
where the last line above is due to Jensen's inequality.

\textit{Step 4}: Let $E_v = \mathbb{E}_{x}[v_{\theta,x}], E_w = \mathbb{E}_x[w_{\theta,x}]$. Rewriting $\mathbb{E}_x[\inner{v_{\theta,x}}{w_{\theta,x}}]$,
\begin{align*}
\mathbb{E}_x[\inner{v_{\theta,x}}{w_{\theta,x}}] &= \mathbb{E}_x[\inner{v_{\theta,x}-E_v + E_v}{w_{\theta,x}-E_w+E_w}] \\
&= \mathbb{E}_x[\inner{v_{\theta,x}-E_v}{w_{\theta,x}-E_w}] + \inner{\mathbb{E}_x[v_{\theta,x}-E_v]}{E_w} \\ 
&+ \inner{E_v}{\mathbb{E}_x[w_{\theta,x}-E_w]} + \inner{E_v}{E_w} \\
&= \mathbb{E}_x[\inner{v_{\theta,x}-E_v}{w_{\theta,x}-E_w}] + \inner{0}{E_w} + \inner{E_w}{0} + \inner{E_v}{E_w} \\
&= \mathbb{E}_x[\inner{v_{\theta,x}-E_v}{w_{\theta,x}-E_w}] + \inner{E_v}{E_w}
\end{align*}
By the Cauchy-Schwarz inequality, the last line above becomes
\begin{align*}
\mathbb{E}_x[\inner{v_{\theta,x}}{w_{\theta,x}}] 
&\leq \sqrt{\mathbb{E}_x[\|v_{\theta,x}-E_v\|^2]}\sqrt{\mathbb{E}_x[\|w_{\theta,x}-E_w\|^2]} + \inner{E_v}{E_w} \\ 
&\leq \sqrt{\text{Var}_x[v_{\theta,x}]}\sqrt{\text{Var}_x[w_{\theta,x}]} + \inner{E_v}{E_w} \\ 
&\leq \text{max}\left(\sqrt{\text{Var}_x[v_{\theta,x}]},\sqrt{\text{Var}_x[w_{\theta,x}]}\right)^2 + \inner{E_v}{E_w} \\ 
&\leq \delta_{var}\|\mathbb{E}_x[M_{\theta}v_{\theta,x}]\|^2 + \inner{E_v}{E_w}
\end{align*}
Rearranging the last line above and applying the result of Step 3 yields
\begin{align}
\inner{E_v}{E_w} &\geq \mathbb{E}_x[\inner{v_{\theta,x}}{w_{\theta,x}}] - \delta_{var}\|\mathbb{E}_x[M_{\theta}v_{\theta,x}]\|^2 \\
&\geq (\lambda_{-}-\gamma\lambda_{+})\|\mathbb{E}_x[M_{\theta}v_{\theta,x}]\|^2 - \delta_{var}\|\mathbb{E}_x[M_{\theta}v_{\theta,x}]\|^2 \\
&\geq (\lambda_{-}-\gamma\lambda_{+} - \delta_{var})\|\mathbb{E}_x[M_{\theta}v_{\theta,x}]\|^2 \\
&= \delta_{\theta}^2 \geq 0
\end{align}
\end{proof}

\subsection{Proof of Lemma~\ref{lemma:expectation_integral_product}}
\textit{Lemma~\ref{lemma:expectation_integral_product}}:
\lemmaExpectationIntegralProduct{app}
\begin{proof}
Let $E_1 = \mathbb{E}_x[\nabla_{\theta}J_x] \text{ and } E_2=\mathbb{E}_x[\dJFB] \text{ so } E_{1}^{\top}E_2 = T^2 \mathbb{E}_x[C_{v}]^{\top}\mathbb{E}_x[C_{w}]$. We have that 
\begin{align}
&\int_{0}^{T}\mathbb{E}_x[v_{\theta,x}(t)]^{\top}\mathbb{E}[w_{\theta,x}(t)]dt \\ &=\int_{0}^{T}\mathbb{E}_x[v_{\theta,x} - C_v + C_v]^{\top}\mathbb{E}_x[w_{\theta,x}(t) - C_w + C_w]dt \\
&= \int_{0}^{T}\mathbb{E}_x[v_{\theta,x}(t) - C_v]^{\top}\mathbb{E}_x[w_{\theta,x}(t) - C_w]dt \\
&+\int_{0}^{T} \mathbb{E}_x[v_{\theta,x}(t) - C_v]^{\top}\mathbb{E}_x[C_w]dt \\
&+\int_{0}^{T}\mathbb{E}_x[C_v]^{\top}\mathbb{E}_x[w_{\theta,x}(t) - C_w] + T\mathbb{E}_x[C_v]^{\top}\mathbb{E}_x[C_w] 
\end{align}
Because $C_w$ does not depend on $t$, we can factor it out of the integral
\begin{equation}
\int_{0}^{T}\mathbb{E}_x[v_{\theta,x}(t) - C_v]^{\top}\mathbb{E}_x[C_w]dt = \left(\int_{0}^{T}\mathbb{E}_x[v_{\theta,x}(t)-C_v]dt\right)^{\top}\mathbb{E}_x[C_w]
\end{equation}
Because $v_{\theta,x}$ is integrable on $[0,T] \times \Omega$ Fubini's Theorem applies and the order of integration can be interchanged, yielding
\begin{align}
\left(\int_{0}^{T}\mathbb{E}_x[v_{\theta,x}(t)-C_v]dt\right)^{\top}\mathbb{E}_x[C_w] &= \left(\mathbb{E}_x\left[\int_{0}^{T}v_{\theta,x}(t)-C_vdt\right]\right)^{\top}\mathbb{E}_x[C_w] \\
&=\left( \mathbb{E}_x[0]\right)^{\top}\mathbb{E}_x[C_w] \\
&= 0,
\end{align}
where we used $\int_0^T (v_{\theta,x}(t) - C_v) dt = 0$. A similar derivation yields
\begin{equation}
\int_{0}^{T} \mathbb{E}_x[w_{\theta,x}(t) - C_w]^{\top}\mathbb{E}_x[C_v]dt = 0
\end{equation}
Therefore, we have
\begin{align}
\int_{0}^{T}\mathbb{E}_x[v_{\theta,x}(t)]^{\top}\mathbb{E}[w_{\theta,x}(t)]dt &= \int_{0}^{T}\mathbb{E}_x[v_{\theta,x}(t) - C_v]^{\top}\mathbb{E}_x[w_{\theta,x}(t) - C_w]dt + T\mathbb{E}_x[C_v]^{\top}\mathbb{E}_x[C_w] \\
\int_{0}^{T}\mathbb{E}_x[v_{\theta,x}(t)]^{\top}\mathbb{E}[w_{\theta,x}(t)]dt &= \int_{0}^{T}\mathbb{E}_x[v_{\theta,x}(t) - C_v]^{\top}\mathbb{E}_x[w_{\theta,x}(t) - C_w]dt + \frac{1}{T}E_{1}^{\top}E_2
\label{eq:eip}
\end{align}
By Cauchy-Schwarz and Assumption~\ref{assumption:expectation_timeaverage}, $\forall t$
\begin{align}
&\mathbb{E}_x[v_{\theta,x}(t) - C_v]^{\top}\mathbb{E}_x[w_{\theta,x}(t) - C_w] \leq \| \mathbb{E}_x[v_{\theta,x}-C_v]\|\|\mathbb{E}_x[w_{\theta,x}-C_w]\| \\
&\leq (a_v +\delta_v \left\|\mathbb{E}_x[\nabla_{\theta}J_x] \right\|) (a_w + \delta_w \left\| \mathbb{E}_x[\dJFB] \right\|) \\
&\leq \max(a_v + \delta_v \left\| \mathbb{E}_x[\nabla_{\theta}J_x ]\right\|, a_w +\delta_w \left\| \mathbb{E}_x[\dJFB] \right\|)^2 \\
&\leq \delta_{\theta}^{2} - \frac{\epsilon_v}{T^2}  \left\|\mathbb{E}_x[\nabla_{\theta}J_x] \right\|^2
\label{eq:delta_theta}
\end{align}
Applying ~\eqref{eq:delta_theta} and algebraically rearranging ~\eqref{eq:eip} yields
\begin{align}
&\int_{0}^{T}\mathbb{E}_x[v_{\theta,x}(t)]^{\top}\mathbb{E}_x[w_{\theta,x}(t)]dt \leq \int_{0}^{T}\delta_{\theta}^{2} - \frac{\epsilon_v}{T^2} \left\| \nabla_{\theta}J_x \right\|^2 dt + \frac{1}{T}E_{1}^{\top}E_2 \\
&E_{1}^{\top}E_2 \geq T\int_{0}^{T}\mathbb{E}_x[v_{\theta,x}(t)]^{\top}\mathbb{E}_x[w_{\theta,x}(t)] -\left( \delta_{\theta}^{2} - \frac{\epsilon_v}{T^2}\|\mathbb{E}_x[\nabla_{\theta}J_x]\|^2 \right)dt 
\end{align}
Thus, by the result of Lemma ~\ref{lemma:expectation_integrand}
\begin{align}
E_{1}^{\top}E_2 &\geq T\int_{0}^{T}\delta_{\theta}^2 -\left( \delta_{\theta}^{2} - \frac{\epsilon_v}{T^2}\|\mathbb{E}_x[\nabla_{\theta}J_x]\|^2 \right)dt \\
E_{1}^{\top}E_2 &\geq T\int_{0}^{T}\frac{\epsilon_v}{T^2}\|\mathbb{E}_x[\nabla_{\theta}J_x]\|^2dt \\
E_{1}^{\top}E_2 &\geq T\left(\frac{\epsilon_v}{T}\|\mathbb{E}_x[\nabla_{\theta}J_x]\|^2\right) \\
E_1^{\top}E_2 &\geq \epsilon_v\|\mathbb{E}_x[\nabla_{\theta}J_x]\|^2
\end{align}
\end{proof}

\subsection{Proof of Lemma~\ref{lemma:expectation_descent}}
\textit{Lemma~\ref{lemma:expectation_descent}}:
\lemmaExpectationDescent{app}
\begin{proof}
Let $F(\theta_j) = \mathbb{E}_x[J_x(\theta_j)]$ and $\nabla_{\theta}F(\theta_j) = \nabla_{\theta}\mathbb{E}_x[J_x(\theta)] = \mathbb{E}_x[\nabla_{\theta}J(\theta_j)]$ since $\mathbb{E}_x[\cdot]$ is an integral with respect to x, not $\theta$. Taking the 2nd order Taylor expansion of $F$ with respect to $\theta$ centered at $\theta_j$, using the fact that $J$, and therefore $F$, is $L_J$-Lipschitz
\begin{equation*}
F(\theta_{j+1}) \leq F(\theta_j) + \nabla_{\theta}F(\theta_j)^{\top}(\theta_{j+1}-\theta_j) + \frac{1}{2}L_{J}\|\theta_{j+1}-\theta_j\|^2
\end{equation*}
Using ~\eqref{eq:SGD_iteration}, we have
\begin{equation*}
F(\theta_{j+1}) -F(\theta_j) \leq -\alpha_j\nabla_{\theta}F(\theta_j)^{\top}\dJFBxi(\theta_j) + \frac{\alpha_j}{2}L_{J}\|\dJFBxi(\theta_j)\|^2.
\label{eq:d1}
\end{equation*}
Taking the conditional expectation with respect to $\xi_j$ on both sides of ~\eqref{eq:d1}, we get
\begin{equation}
\begin{aligned}
& \mathbb{E}_{\xi_j}[F(\theta_{j+1})] -F(\theta_j) \\ &\leq 
\mathbb{E}_{\xi_j}[-\alpha_j\nabla_{\theta}F(\theta_j)^{\top}\dJFBxi(\theta_j)] + \mathbb{E}_{\xi_j}[\frac{1}{2}\alpha_j^2L_{J}\|\dJFBxi(\theta_j)\|^2]
\\
&= -\alpha_j\nabla_{\theta}F(\theta_j)^{\top}\mathbb{E}_{\xi_j}[\dJFBxi(\theta_j)] + \mathbb{E}_{\xi_j}[\frac{1}{2}\alpha_j^2L_{J}\|\dJFBxi(\theta_j)\|^2].
\end{aligned}
\label{eq:d2}
\end{equation}
Note that, $\theta_j$ depends only on $\xi_{j-1},...,\xi_1,\xi_0$ generated in previous iterates and $\mathbb{E}_{\xi_j}[F(\theta_{j})] = F(\theta_{j})$. 
By Cauchy-Schwarz inequality,
\[
\left\|\int_{0}^{T}M_{\theta_j}^{\top}h_{\theta, x} dt\right\|^2 \leq \left(\int_0^T 1^2 dt\right) \left(\int_0^T \|M_{\theta_j}^{\top}h_{\theta, x}\|^2 dt\right) = T \int_0^T \|M_{\theta_j}^{\top}h_{\theta, x}\|^2 dt \leq T \int_0^T \|M_{\theta_j}^{\top}\|^2\|h_{\theta, x}\|^2 dt.
\]
Thus, by Assumptions ~\ref{assumption:h}-~\ref{assumption:M},
\begin{align}
\mathbb{E}_{\xi_j}[\|d^\mathrm{JFB}_{\xi_j}(\theta_j)\|^2] = \mathbb{E}\left[\left\|\frac{1}{B} \sum_{b=1}^B \dJFBb(\theta_j)\right\|^2\right] \leq
\mathbb{E}_x[\|\dJFB(\theta_j)\|^2] 
&\leq \frac{B_{max}^2T^2}{\beta},
\label{eq:ub_d_xi_j}
\end{align}
where we remind the reader that $d^\mathrm{JFB}_{\xi_j}(\theta_j)$ may correspond to the JFB gradient computed on a minibatch of samples and $\dJFB(\theta_j)$ corresponds to the JFB gradient computed on a single sample.

Thus, by the result of Lemma ~\ref{lemma:expectation_integral_product}, ~\eqref{eq:d2}, ~\eqref{eq:ub_d_xi_j}, and the fact that $\mathbb{E}_{\xi_j}[\dJFBxi(\theta_j)] = \mathbb{E}_{x}[\dJFB(\theta_j)]$, it holds that
\begin{equation}
\mathbb{E}_{\xi_j}[\mathbb{E}_x[J_x(\theta_{j+1})]] -\mathbb{E}_x[J_x(\theta_j)] \leq 
-\alpha_j\epsilon_v\|\mathbb{E}_x[\nabla_{\theta}J_x(\theta_j)]\|^2 + \frac{\alpha_{j}^2L_JB_{max}^2T^2}{2\beta}
\label{eq:d4}
\end{equation}
For the RHS of ~\eqref{eq:d4} to be $\leq 0$, it must be true that, with $\alpha_j > 0$
\begin{align}
 \frac{\alpha_{j}^2L_JB_{max}^2T^2}{2\beta} &\leq \alpha_j\epsilon_v\|\mathbb{E}_x[\nabla_{\theta}J_x(\theta_j)]\|^2 \\
 \alpha_j &\leq \frac{2\beta\epsilon_v}{L_JB_{max}^2T^2}\|\mathbb{E}_x[\nabla_{\theta}J_x(\theta_j)]\|^2
\end{align}
The following upper bound, uniform in $j$, for $\|\mathbb{E}_x[\nabla_{\theta}J_x(\theta_j)]\|^2$ can be derived in a very similar manner to that of $\mathbb{E}_x[\|\dJFB(\theta_j)\|^2]$ above,
\begin{equation}
\|\mathbb{E}_x[\nabla_{\theta}J_x(\theta_j)]\|^2 \leq \frac{B_{max}^2T^2}{\beta(1-\gamma)^2}
\end{equation}
Thus, for the RHS of \eqref{eq:d4} to be $\leq 0$, it is necessary that
\begin{align}
\alpha_j \leq \frac{2\epsilon_v}{L_J(1-\gamma)^2}
\end{align}
\end{proof}

\subsection{Proof of Theorem~\ref{theorem:expectation_cesaro_convergence}}
\textit{Theorem~\ref{theorem:expectation_cesaro_convergence}}:
\theoremExpectationCesaroConvergence{app}
\begin{proof}
Taking the total expectation of ~\eqref{eq:d4}, we have
\begin{equation}
\mathbb{E}[\mathbb{E}_x[J_x(\theta_{j+1})]] - \mathbb{E}[\mathbb{E}_x[J_x(\theta_j)]] \leq \\
-\alpha_j\epsilon_v\mathbb{E}[\|\mathbb{E}_{x}[\nabla_{\theta}J_x(\theta_j)]\|^2] + \frac{\alpha^2L_JB_{max}^2T^2}{2\beta}
\label{eq:conv1}
\end{equation}
Setting $j=0$ in ~\eqref{eq:conv1}, we have
\begin{align}
&\mathbb{E}[\mathbb{E}_x[J_x(\theta_{1})]] \leq \mathbb{E}[\mathbb{E}_x[J_x(\theta_0)]]
-\alpha_j\epsilon_v\mathbb{E}[\|\mathbb{E}_{x}[\nabla_{\theta}J_x(\theta_0)]\|^2] + \frac{\alpha_j^2L_JB_{max}^2T^2}{2\beta}
\label{eq:conv2}
\end{align}
Setting $j=1$ in ~\eqref{eq:conv1} and applying ~\eqref{eq:conv2} 
\begin{align*}
&\mathbb{E}[\mathbb{E}_x[J_x(\theta_{2})]] \leq \mathbb{E}[\mathbb{E}_x[J_x(\theta_1)]]
-\alpha_j\epsilon_v\mathbb{E}[\|\mathbb{E}_{x}[\nabla_{\theta}J_x(\theta_1)]\|^2] + \frac{\alpha_j^2L_JB_{max}^2T^2}{2\beta}\\
&\mathbb{E}[\mathbb{E}_x[J_x(\theta_{2})]] \leq \mathbb{E}[\mathbb{E}_x[J_x(\theta_0)]]
-\epsilon_v\sum_{j=0}^{1}\alpha_j\mathbb{E}[\|\mathbb{E}_{x}[\nabla_{\theta}J_x(\theta_j)]\|^2] + \frac{L_JB_{max}^2T^2}{2\beta}\sum_{j=0}^{1}\alpha_{j}^2\\
&\mathbb{E}[\mathbb{E}_x[J_x(\theta_{2})]] - \mathbb{E}[\mathbb{E}_x[J_x(\theta_0)]]\leq 
-\epsilon_v\sum_{j=0}^{1}\alpha_j\mathbb{E}[\|\mathbb{E}_{x}[\nabla_{\theta}J_x(\theta_j)]\|^2] + \frac{L_JB_{max}^2T^2}{2\beta}\sum_{j=0}^{1}\alpha_{j}^2 \\
&\vdots  \\
&\mathbb{E}[\mathbb{E}_x[J_x(\theta_{K})]] - \mathbb{E}[\mathbb{E}_x[J_x(\theta_0)]]\leq 
-\epsilon_v\sum_{j=0}^{K-1}\alpha_j\mathbb{E}[\|\mathbb{E}_{x}[\nabla_{\theta}J_x(\theta_j)]\|^2] + \frac{L_JB_{max}^2T^2}{2\beta}\sum_{j=0}^{K-1}\alpha_{j}^2
\end{align*}
Since $J_x$ is bounded from below by $J_{inf}$ by Assumption ~\ref{assumption:T}, algebraically rearranging the last line above yields
\begin{align}
&\sum_{j=0}^{K-1}\alpha_j\mathbb{E}[\|\mathbb{E}_{x}[\nabla_{\theta}J_x(\theta_j)]\|^2] \leq 
\frac{\mathbb{E}[\mathbb{E}_x[J_x(\theta_0)]] - J_{inf}}{\epsilon_v}+ \frac{L_JB_{max}^2T^2}{2\beta\epsilon_v}\sum_{j=0}^{K-1}\alpha_{j}^2\\
&\frac{1}{A_K}\sum_{j=0}^{K-1}\alpha_j\mathbb{E}[\|\mathbb{E}_{x}[\nabla_{\theta}J_x(\theta_j)]\|^2] \leq \frac{\mathbb{E}[\mathbb{E}_x[J_x(\theta_0)]] - J_{inf}}{\epsilon_vA_K}+ \frac{L_JB_{max}^2T}{2\beta\epsilon_vA_K}\sum_{j=0}^{K-1}\alpha_{j}^2
\label{eq:conv3}
\end{align}
Using linearity of expectation, ~\eqref{eq:conv3} becomes
\begin{equation}
\mathbb{E}\left[\frac{1}{A_K}\sum_{j=0}^{K-1}\alpha_j\|\mathbb{E}_{x}[\nabla_{\theta}J_x(\theta_j)]\|^2\right] \leq \frac{\mathbb{E}[\mathbb{E}_x[J_x(\theta_0)]] - J_{inf}}{\epsilon_vA_K}+ \frac{L_JB_{max}^2T^2}{2\beta\epsilon_vA_K}\sum_{j=0}^{K-1}\alpha_{j}^2
\label{eq:conv4}
\end{equation}
Thus, since $\lim_{K \rightarrow \infty} A_K = \infty$, we have
\begin{align*}
&\lim_{K \rightarrow \infty}\mathbb{E}\left[\frac{1}{A_K}\sum_{j=0}^{K}\alpha_j\|\mathbb{E}_x[\nabla_{\theta}J_x(\theta_j)]\|^2\right] \\
&\leq \lim_{K \rightarrow \infty}\left[\frac{2\beta(\mathbb{E}[\mathbb{E}_x[J_x(\theta_0)]]-J_{inf}) + L_{J}B_{max}^2T^2L{\sum_{j=0}^{K-1}\alpha_{j}^2}}{2 \beta \epsilon_vA_K}\right] \\
&= 0
\end{align*}
\end{proof}

\subsection{Proof of Theorem~\ref{theorem:expectation_liminf}}
\textit{Theorem~\ref{theorem:expectation_liminf}}:
\theoremExpectationLiminf{app}
\begin{proof}
Suppose, for contradiction that, for some $a > 0$
\begin{equation}
\liminf_{j \rightarrow \infty} \mathbb{E}\left[\left\| \mathbb{E}_{x}[\nabla_{\theta}J_{x}(\theta_j)]\right\|^2 \right] = a
\end{equation}
Let $K \in \mathbb{N}$ and $A_K = \sum_{j=0}^{K-1}\alpha_j$. Then, using linearity of expectation
\begin{equation*}
\mathbb{E}\left[ \frac{1}{A_K}\sum_{j=0}^{K-1}\alpha_j \left\| \mathbb{E}_{x}[\nabla_{\theta}J_{x}(\theta_j)] \right\|^2 \right] = \frac{1}{A_K}\sum_{j=0}^{K-1}\alpha_j\mathbb{E}\left[ \left\| \mathbb{E}_{x}[\nabla_{\theta}J_{x}(\theta_j)] \right\|^2\right]
\end{equation*}
Taking the liminf as $K \rightarrow \infty$ on both sides of the above equation yields a contradiction, as the LHS converges to 0 but the RHS diverges. Hence, by contradiction the result follows. 
\end{proof}

\subsection{Proof of Corollary~\ref{corollary:convergence_probability}}
\textit{Corollary~\ref{corollary:convergence_probability}}:
\corollaryConvergenceProbability{app}
\begin{proof}
This proof is the same as that of Theorem 4.11 in in \cite{bottou2018optimization} but we will include it here to be complete. Let $\epsilon > 0$ and let $\mathbb{E}[\cdot]$ represent total expectation. By Markov's inequality and the law of total expectation, also known as the tower property,
\begin{align}
P(\|\mathbb{E}_x[\nabla_{\theta}J_x(\theta_{j(K)})]\| &\geq \epsilon) = P(\|\mathbb{E}_x[\nabla_{\theta}J_x(\theta_{j(K)})]\|^2 \geq \epsilon^2) \\
&\leq \frac{1}{\epsilon^2}\mathbb{E}[\mathbb{E}_{j(K)}[\|\mathbb{E}_x[\nabla_{\theta}J_x(\theta_{j(K)})]\|^2]]
\label{eq:convergence_in_prob}
\end{align}
By the proof of Theorem ~\ref{theorem:expectation_cesaro_convergence} we have $\lim_{K \rightarrow \infty}\mathbb{E} \left[\sum_{j=0}^{K-1}\alpha_j \|\mathbb{E}_x[\nabla_{\theta}J_x(\theta_j)]\|^2 \right] < \infty$. Therefore, we must have $\lim_{j \rightarrow \infty}\mathbb{E}\left[\alpha_j\|\mathbb{E}_x[\nabla_{\theta}J_x(\theta_j)]\|^2\right] = 0$. Thus, by ~\eqref{eq:convergence_in_prob},
\begin{equation}
\lim_{K \rightarrow \infty} P(\|\mathbb{E}_x[\nabla_{\theta}J_x(\theta_{j(K)})]\| \geq \epsilon) \leq \lim_{K \rightarrow \infty } \frac{1}{\epsilon^2}\mathbb{E}[\mathbb{E}_{j(K)}[\|\mathbb{E}_x[\nabla_{\theta}J_x(\theta_{j(K)})]\|^2]] = 0
\end{equation}
Since the choice of $\epsilon > 0$ was arbitrary, this holds $\forall \epsilon > 0$, proving convergence in probability.
\end{proof}

\section{Problem Formulation for Optimal Consumption Savings}
\label{app:consumption_description}
We consider the consumption optimization of multiple investors. This is adapted from \cite{detemple1992optimal, angoshtari2023optimal}. The wealth $x_t$, $t \in [0,T]$ of an individual follows
    \begin{equation*}
        \frac{d{x}_t}{dt} = r x_t - \mathbf{1}^\top u_t,
    \end{equation*}
with given initial condition $x_0$. Here $r$ is the risk-free rate, the individual choose a consumption strategy $u_t \in \mathbb{R}^m$ for $m$ products, while tracking a habit level $h_t \in \mathbb{R}^m$, with given initial condition $h_0$,
      \begin{equation*}
        \frac{d{h}(t)}{dt} = A u^{\circ \eta}(t) - B h^{\circ \theta}(t),
    \end{equation*}
    where $\circ$ denotes component-wise power. To simplify the setup, we take $B$ to be diagonal with entries $B_i \geq 0$.
    An individual maximizes a CRRA-type discounted utility
\begin{equation*}
    \int_0^T e^{-\delta t} \frac{\sum_{i=1}^n \left(u_i(t) - h_i(t)\right)^{1-\gamma}}{1 - \gamma} dt + e^{-\delta T} \epsilon \frac{x_T^{1-\gamma}}{1-\gamma},
\end{equation*}
where $u_i(t) > h_i(t)$ for $i = 1, \ldots, m$, and $\gamma \neq 1$, $\epsilon$, $\delta$, $\eta$, $\theta$ are positive constants. We use a reparametrization to ensure the constraint $u_i(t) - h_i(t) > 0$ for all $i$ in experiments. The consumption behavior exhibits habit formation, that is, the current utility depends on the consumption relative to an individual's habit level established from past consumption. Then, the Hamiltonian is
\begin{equation*}
    \mathcal{H}(t, x, h, u, p) = 
    e^{-\delta t} \sum_{i=1}^m \frac{(u_i - h_i)^{1-\gamma}}{1 - \gamma} + p_x (r x - \mathbf{1}^\top u) + p_h^\top (A u^{\circ \eta} - B h^{\circ \theta}).
\end{equation*}
Let $\frac{\partial \mathcal{H}}{\partial u} = 0$, then $e^{-\delta t} (u - h)^{-\gamma} - p_x + \alpha \eta p_h u^{\eta-1} = 0$. Except when $\eta = \theta =1$, this first-order condition does not have a closed-form solution in general.

\section{Proof of Convergence of SGD with JFB as Stochastic Gradient to a Neighborhood of a Critical Point Under Weaker Assumptions}
\label{app:convergence_to_nbhd}
It is possible to prove convergence in expectation of JFB-based SGD using the iteration ~\eqref{eq:SGD_iteration} to a neighborhood of a critical point of $J$ under weaker assumptions than Theorem ~\ref{theorem:expectation_cesaro_convergence} depends on. This section was inspired by \cite{demidovich2024guide}.

\begin{assumption}
    \label{assumption:app_weak_bound_on_integrand_var}
    Modify Assumption ~\ref{assumption:expectation_integrand_inner_product} to be $\exists 0 < \delta_{var} < \lambda_{-} - \gamma \lambda_{+}$ and $\epsilon_1 \geq 0$ such that $\forall z, u, \theta$
    \begin{equation*}
    \text{max}(\text{Var}_x[v_{\theta,x}(t)],\text{Var}_x[w_{\theta,x}(t)])^2 \leq \epsilon_1 + \delta_{var}\|\mathbb{E}_x[M_{\theta}v_{\theta,x}]\|^2
    \end{equation*}
\end{assumption}

\begin{lemma}
Under the Assumptions of ~\ref{assumption:T}-~\ref{assumption:M} and Assumption ~\ref{assumption:app_weak_bound_on_integrand_var}, $\forall z, u, \theta$
\begin{equation}
\langle \mathbb{E}_x[v_{\theta,x}(t)],\mathbb{E}_x[w_{\theta,x}(t)] \rangle \geq \delta_{\theta}^2 - \epsilon_1
\end{equation}
where $\delta_{\theta}$ is defined in Assumption ~\ref{assumption:expectation_timeaverage} and $\epsilon_1$ is defined in Assumption ~\ref{assumption:app_weak_bound_on_integrand_var}
\label{lemma:expectation_integrand_weak}
\end{lemma}
\begin{proof}
Let $E_v = \mathbb{E}_x[v_{\theta,x}(t)], \ E_w = \mathbb{E}_x[w_{\theta,x}(t)]$. With these assumptions and this notation, everything up to the line
\begin{equation}
\mathbb{E}_x[\langle v_{\theta,x},w_{\theta,x} \rangle] \leq \text{max}\left(\sqrt{\text{Var}_{x}[v_{\theta,x}]},\sqrt{\text{Var}_{x}[w_{\theta,x}]}\right)^2 + \langle E_v,E_w\rangle
\end{equation}
in the proof of Lemma ~\ref{lemma:expectation_integrand} still holds. Applying Assumption ~\ref{assumption:app_weak_bound_on_integrand_var} to this equation gives
\begin{align}
\mathbb{E}_x[\langle v_{\theta,x},w_{\theta,x} \rangle] &\leq \epsilon_1 + \delta_{var}\|\mathbb{E}_x[v_{\theta,x}]\|^2 + \langle E_v,E_w\rangle \\
\langle E_v,E_w\rangle &\geq \mathbb{E}_x[\inner{v_{\theta,x}}{w_{\theta,x}}] - \delta_{var}\|\mathbb{E}_x[v_{\theta,x}]\|^2 - \epsilon_1\\
\inner{E_v}{E_w} &\geq (\lambda_{-} - \gamma\lambda_{+})\|\mathbb{E}_x[M_{\theta}v_{\theta,x}]\|^2 - \delta_{var}\|\mathbb{E}_x[v_{\theta,x}]\|^2 - \epsilon_1 \\
\langle E_v,E_w\rangle &\geq (\lambda_{-} - \gamma \lambda_{+} - \delta_{var})\|\mathbb{E}_x[M_{\theta}v_{\theta,x}]\|^2 - \epsilon_1 \\ 
\langle E_v,E_w\rangle &\geq \delta_{\theta}^2 - \epsilon_1
\end{align}
\end{proof}

\begin{lemma}
Under Assumptions ~\ref{assumption:T}-~\ref{assumption:M}, ~\ref{assumption:app_weak_bound_on_integrand_var}, and ~\ref{assumption:expectation_timeaverage}, $\forall u,z,\theta$
\begin{equation*}
\mathbb{E}_x[\nabla_{\theta}J_x]^{\top}\mathbb{E}_x[\dJFB] \geq \epsilon_v\|\mathbb{E}_x[\nabla_{\theta}J_x]\|^2 - \epsilon_1'
\end{equation*}
where $\epsilon_v$ is given in Assumption ~\ref{assumption:expectation_timeaverage} and $\epsilon_1' = T^2\epsilon_1$ with $\epsilon_1$ is from Assumption ~\ref{assumption:app_weak_bound_on_integrand_var}
\label{lemma:weak_expectation_integral_product}
\end{lemma}
\begin{proof}
Under these assumptions, everything in the proof of Lemma ~\ref{lemma:expectation_integral_product} up to and including the line
\begin{equation}
E_{1}^{\top}E_2 \geq T\int_{0}^{T}\mathbb{E}_x[v_{\theta,x}(t)]^{\top}\mathbb{E}_x[w_{\theta,x}(t)] -\left( \delta_{\theta}^{2} - \frac{\epsilon_v}{T^2}\|\mathbb{E}_x[\nabla_{\theta}J_x]\|^2 \right)dt 
\end{equation}
still holds. Applying the result of Lemma ~\ref{lemma:expectation_integrand_weak}, we obtain
\begin{align}
E_{1}^{\top}E_2 &\geq T\int_{0}^{T}\delta_{\theta}^2 - \epsilon_1 -\left( \delta_{\theta}^{2} - \frac{\epsilon_v}{T^2}\|\mathbb{E}_x[\nabla_{\theta}J_x]\|^2 \right)dt \\
E_{1}^{\top}E_2 &\geq T\int_{0}^{T}\frac{\epsilon_v}{T^2}\|\mathbb{E}_x[\nabla_{\theta}J_x]\|^2 -\epsilon_1dt \\
E_{1}^{\top}E_2 &\geq \epsilon_v\|\mathbb{E}_x[\nabla_{\theta}J_x]\|^2 - T^2\epsilon_1 = \epsilon_v\|\mathbb{E}_x[\nabla_{\theta}J_x]\|^2 - \epsilon_1'
\end{align}
\end{proof}
\begin{theorem}
\label{theorem:appendix_theorem}
Under the assumptions of Lemma ~\ref{lemma:weak_expectation_integral_product}, the JFB-based SGD iteration ~\eqref{eq:SGD_iteration} with constant step size/learning rate $\alpha > 0$, will converge in (total) expectation to a neighborhood of the critical point
\begin{equation}
 \lim_{K \rightarrow \infty}\mathbb{E} \left[\frac{1}{K} \sum_{j=1}^{K}\|\mathbb{E}_x[\nabla_{\theta}J_x]\|^2\right] \leq \frac{\alpha^2 L_J B_{max}^2T^2}{2 \beta} + \frac{\epsilon_1'}{\epsilon_v}
\end{equation}
\end{theorem}
\begin{proof}
Under these assumptions, everything in the proof of Lemma ~\ref{lemma:expectation_descent} up to and including the line
\begin{equation}
\mathbb{E}_{\xi_j}[F(\theta_{j+1})] - F(\theta_j) \leq -\alpha_j\nabla_{\theta}F(\theta_j)^{\top}\mathbb{E}_{\xi_j}[d_{\xi_j}^{JFB}(\theta_j)] + \mathbb{E}_{\xi_j}\left[\frac{\alpha_{j}^2L_J}{2}\|\dJFB(\theta_j)\|^2\right]
\end{equation}
still holds. Replacing $F$ with $\mathbb{E}_x[J_x]$, $\mathbb{E}_{\xi_j}[d_{\xi_j}^{JFB}(\theta_j)]$ with $\mathbb{E}_x[\dJFB(\theta_j)]$ like in Lemma ~\ref{lemma:expectation_descent}, and applying the result of Lemma ~\ref{lemma:weak_expectation_integral_product} and the upper bound $\mathbb{E}_{\xi_j}[\|d_{\xi_j}^{JFB}(\theta)\|^2] \leq \frac{B_{max}^2T^2}{\beta}$, we have
\begin{equation}
\mathbb{E}_{\xi_j}[\mathbb{E}_x[J_x(\theta_{j+1})]] -\mathbb{E}_x[J_x(\theta_j)] \leq -\alpha(\epsilon_v\|\mathbb{E}_x[\nabla_{\theta}J_x(\theta_j)]\|^2-\epsilon_1') + \frac{\alpha^2}{2}L_J\left(\frac{B_{max}^2T^2}{\beta}\right)
\end{equation}
Taking total expectation of the above line, we have
\begin{equation}
\mathbb{E}[\mathbb{E}_x[J_x(\theta_{j+1})]] - \mathbb{E}[\mathbb{E}_x[J_x(\theta_j)]] \leq -\alpha(\epsilon_v\mathbb{E}[\|\mathbb{E}_{x}[\nabla_{\theta}J_x(\theta_j)]\|^2] - \epsilon_1') + \frac{\alpha^2}{2}L_{J}\left( \frac{B_{max}^2T^2}{\beta}\right)
\label{eq:descent1}
\end{equation}
Setting $j=0$ in ~\eqref{eq:descent1}, we have
\begin{align}
\mathbb{E}[\mathbb{E}_x[J_x(\theta_{1})]] - \mathbb{E}[\mathbb{E}_x[J_x(\theta_0)]] &\leq -\alpha(\epsilon_v\mathbb{E}[\|\mathbb{E}_{x}[\nabla_{\theta}J_x(\theta_0)]\|^2] - \epsilon_1') + \frac{\alpha^2}{2}L_{J}\left( \frac{B_{max}^2T^2}{\beta}\right) \\
\mathbb{E}[\mathbb{E}_x[J_x(\theta_{1})]] &\leq \mathbb{E}[\mathbb{E}_x[J_x(\theta_0)]] -\alpha(\epsilon_v\mathbb{E}[\|\mathbb{E}_{x}[\nabla_{\theta}J_x(\theta_0)]\|^2] - \epsilon_1') + \frac{\alpha^2}{2}L_{J}\left( \frac{B_{max}^2T^2}{\beta}\right)
\label{eq:descent2}
\end{align}
Setting $j=1$ in ~\eqref{eq:descent1} and applying ~\eqref{eq:descent2} yields
\begin{align}
\mathbb{E}[\mathbb{E}_x[J_x(\theta_{2})]] - \mathbb{E}[\mathbb{E}_x[J_x(\theta_1)]] &\leq -\alpha(\epsilon_v\mathbb{E}[\|\mathbb{E}_{x}[\nabla_{\theta}J_x(\theta_1)]\|^2] - \epsilon_1') + \frac{\alpha^2}{2}L_{J}\left( \frac{B_{max}^2T^2}{\beta}\right) \\
\label{eq:descent3}
\mathbb{E}[\mathbb{E}_x[J_x(\theta_{2})]] &\leq \mathbb{E}[\mathbb{E}_x[J_x(\theta_1)]] -\alpha(\epsilon_v\mathbb{E}[\|\mathbb{E}_{x}[\nabla_{\theta}J_x(\theta_1)]\|^2] - \epsilon_1') + \frac{\alpha^2}{2}L_{J}\left( \frac{B_{max}^2T^2}{\beta}\right) \\
\mathbb{E}[\mathbb{E}_x[J_x(\theta_{2})]] &\leq \mathbb{E}[\mathbb{E}_x[J_x(\theta_0)]] -\alpha(\epsilon_v\mathbb{E}[\|\mathbb{E}_{x}[\nabla_{\theta}J_x(\theta_0)]\|^2] - \epsilon_1') + \frac{\alpha^2}{2}L_{J}\left( \frac{B_{max}^2T^2}{\beta}\right) - \\ &\alpha(\epsilon_v\mathbb{E}[\|\mathbb{E}_{x}[\nabla_{\theta}J_x(\theta_1)]\|^2] - \epsilon_1') + \frac{\alpha^2}{2}L_{J}\left( \frac{B_{max}^2T^2}{\beta}\right) \\
\mathbb{E}[\mathbb{E}_x[J_x(\theta_{2})]] - \mathbb{E}[\mathbb{E}_x[J_x(\theta_0)]] &\leq -\alpha\sum_{j=0}^{1}(\epsilon_v\mathbb{E}[\|\mathbb{E}_{x}[\nabla_{\theta}J_x(\theta_j)]\|^2] - \epsilon_1') + (2)\frac{\alpha^2}{2}L_{J}\left( \frac{B_{max}^2T^2}{\beta}\right) \\
& \vdots \\
\mathbb{E}[\mathbb{E}_x[J_x(\theta_{K})]] - \mathbb{E}[\mathbb{E}_x[J_x(\theta_0)]] &\leq -\alpha\sum_{j=0}^{K-1}(\epsilon_v\mathbb{E}[\|\mathbb{E}_{x}[\nabla_{\theta}J_x(\theta_j)]\|^2] - \epsilon_1') + K\left(\frac{\alpha^2 L_J B_{max}^2 T}{2 \beta}\right)
\label{eq:descent4}
\end{align}
Because $J_x(\theta)$ is bounded from below by $J_{inf}$ (restrict domain if necessary), \eqref{eq:descent4} becomes, using linearity of expectation,
\begin{align*}
\alpha \epsilon_v \mathbb{E}\left[\sum_{j=0}^{K-1}\|\mathbb{E}_x[\nabla_{\theta}J_x]\|^2\right] &\leq \mathbb{E}[\mathbb{E}_x[\nabla_{\theta}J_x(\theta_0)]] - J_{inf} + \alpha K \left( \frac{\alpha B_{max}^2 T L_J}{2 \beta} + \epsilon_1'\right) \\
\mathbb{E}\left[\frac{1}{K}\sum_{j=0}^{K-1}\|\mathbb{E}_x[\nabla_{\theta}J_x]\|^2\right] &\leq \frac{\mathbb{E}[\mathbb{E}_x[\nabla_{\theta}J_x(\theta_0)]] - J_{inf}}{K\alpha \epsilon_v} + \left( \frac{\alpha B_{max}^2 T L_J}{2 \beta \epsilon_v} + \frac{\epsilon_1'}{\epsilon_v}\right)
\end{align*}
Thus, 
\begin{align}
\lim_{K \rightarrow \infty} \mathbb{E}\left[ \frac{1}{K}\sum_{j=0}^{K}\|\mathbb{E}_x[\nabla_{\theta}J_x(\theta_j)]\|^2 \right]
& \leq \lim_{K \rightarrow \infty} \left[\frac{\mathbb{E}[\mathbb{E}_x[\nabla_{\theta}J_x(\theta_0)]] - J_{inf}}{K\alpha \epsilon_v} + \frac{\alpha B_{max}^2 T L_J}{2 \beta \epsilon_v} + \frac{\epsilon_1'}{\epsilon_v} \right] \\
&= \frac{\alpha L_J B_{max}^2 T^2}{2 \beta \epsilon_v} + \frac{\epsilon_1'}{\epsilon_v}
\end{align}
\end{proof}
It can be seen from the result of Theorem ~\ref{theorem:appendix_theorem} that as $\alpha \rightarrow 0$, $\frac{\alpha L_J B_{max}^2T^2}{2 \beta \epsilon_v} \rightarrow 0$, leaving only $\frac{\epsilon_1'}{\epsilon_v}$. Thus, even if a decreasing step size is used, like in Theorem \ref{theorem:expectation_cesaro_convergence}, it is impossible to achieve
\begin{equation}
    \liminf_{j \rightarrow \infty}\mathbb{E}\left[ \|\mathbb{E}_x[\nabla_{\theta}J_x(\theta_j)]\|^2\right] = 0
\end{equation}
under Assumption ~\ref{assumption:app_weak_bound_on_integrand_var}. The JFB-based SGD iteration ~\eqref{eq:SGD_iteration} will only be able to, at best, in expectation, get to a neighborhood of radius $\frac{\epsilon_1'}{\epsilon_v}$ of a critical point of $J$.

\section{Convergence to Critical Point Under Slightly Stronger Version of Assumption ~\ref{assumption:app_weak_bound_on_integrand_var}}
With a slight modification to Assumption ~\ref{assumption:app_weak_bound_on_integrand_var}, it is possible to prove the convergence of the iteration ~\eqref{eq:SGD_iteration} to a critical point, not just to a neighborhood. This assumption is marginally stronger than Assumption ~\ref{assumption:app_weak_bound_on_integrand_var} but is significantly weaker than Assumption ~\ref{assumption:expectation_integrand_inner_product}.
\begin{assumption}
    \label{assumption:limiting_bound_on_integrand_var}
    Modify Assumption ~\ref{assumption:expectation_integrand_inner_product} to be $\exists 0 < \delta_{var} < \lambda_{-} - \gamma \lambda_{+}$ and a sequence of $\{\theta_j\}_{j=0}^{\infty}$ such that, for a nonnegative sequence $\{\epsilon_j\}_{j=0}^{\infty}$ that satisfies $\sum_{j=0}^{\infty}\epsilon_j < \infty$, and $\forall z, u, j$
    \begin{equation*}
    \text{max}(\text{Var}_x[v_{\theta_j,x}(t)],\text{Var}_x[w_{\theta_j,x}(t)])^2 \leq \epsilon_j + \delta_{var}\|\mathbb{E}_x[M_{\theta_j}v_{\theta_j,x}]\|^2
    \end{equation*}
\end{assumption}
\begin{lemma}
Under the Assumptions of ~\ref{assumption:T}-~\ref{assumption:M} and Assumption ~\ref{assumption:limiting_bound_on_integrand_var}, $\forall z, u,j$
\begin{equation}
\langle \mathbb{E}_x[v_{\theta_j,x}(t)],\mathbb{E}_x[w_{\theta_j,x}(t)] \rangle \geq \delta_{\theta_j}^2 - \epsilon_j
\end{equation}
where $\delta_{\theta}$ is defined in Assumption ~\ref{assumption:expectation_timeaverage} and $\epsilon_j$ is defined in Assumption ~\ref{assumption:limiting_bound_on_integrand_var}
\label{lemma:expectation_integrand_limiting}
\end{lemma}
The proof of this lemma is essentially identical to that of Lemma ~\ref{lemma:expectation_integrand_weak} and is thus omitted. \newline
To prove convergence under Assumption ~\ref{assumption:limiting_bound_on_integrand_var}, we will need the following lemma, whose proof is omitted because it is essentially identical to that of Lemma ~\ref{lemma:weak_expectation_integral_product}.
\begin{lemma}
Under Assumptions ~\ref{assumption:T}-~\ref{assumption:M}, ~\ref{assumption:limiting_bound_on_integrand_var}, and ~\ref{assumption:expectation_timeaverage}, $\forall u,z,j$
\begin{equation*}
\mathbb{E}_x[\nabla_{\theta_j}J_x]^{\top}\mathbb{E}_x[\dJFB] \geq \epsilon_v\|\mathbb{E}_x[\nabla_{\theta_j}J_x]\|^2 - \epsilon_j'
\end{equation*}
where $\epsilon_v$ is given in Assumption ~\ref{assumption:expectation_timeaverage} and $\epsilon_j' = T^2\epsilon_j$ with $\epsilon_j$ is from Assumption ~\ref{assumption:limiting_bound_on_integrand_var}
\label{lemma:limiting_expectation_integral_product}
\end{lemma}
We can now prove convergence to the critical point under Assumption ~\ref{assumption:limiting_bound_on_integrand_var}. 

\begin{theorem}
\label{theorem:limiting_theorem}
Suppose the hypotheses of Lemma ~\ref{lemma:limiting_expectation_integral_product} are true with the sequence $\{\epsilon_j\}_{j=0}^{\infty}$ from Assumption ~\ref{assumption:limiting_bound_on_integrand_var} satisfying $\epsilon_0 < \frac{\epsilon_vB_{max}^2}{\beta(1-\gamma)^2} $. Then, the JFB-based SGD iteration ~\eqref{eq:SGD_iteration} with decreasing step size/learning rate  $\alpha_j > 0, \sum_{j=0}^{\infty}\alpha_j = \infty, \sum_{j=0}^{\infty}\alpha_j^2 < \infty$, and $\alpha_j \leq \frac{2 \epsilon_v}{L_J(1-\gamma)^2} - \frac{2\epsilon_j' \beta}{L_JB_{max}^2T^2} \forall j$ will converge in (total) expectation to a critical point. In other words,
\begin{equation}
 \lim_{K \rightarrow \infty}\mathbb{E} \left[\frac{1}{K} \sum_{j=1}^{K}\|\mathbb{E}_x[\nabla_{\theta}J_x]\|^2\right] = 0
\end{equation}
\end{theorem}

\begin{proof}
To prove this theorem, an approach identical to the beginning of the proof of Theorem ~\ref{theorem:appendix_theorem} yields the expression below $\forall j$
\begin{align}
\mathbb{E}[\mathbb{E}_x[J_x(\theta_{j+1})]] - \mathbb{E}[\mathbb{E}_x[J_x(\theta_j)]] &\leq -\alpha_j(\epsilon_v\mathbb{E}[\|\mathbb{E}_{x}[\nabla_{\theta}J_x(\theta_j)]\|^2] - \epsilon_j') + \frac{\alpha_j^2}{2}L_{J}\left( \frac{B_{max}^2T^2}{\beta}\right) \\
\mathbb{E}[\mathbb{E}_x[J_x(\theta_{j+1})]] - \mathbb{E}[\mathbb{E}_x[J_x(\theta_j)]] &\leq -\alpha_j\epsilon_v\mathbb{E}[\|\mathbb{E}_{x}[\nabla_{\theta}J_x(\theta_j)]\|^2] + \alpha_j\epsilon_j' + \frac{\alpha_j^2}{2}L_{J}\left( \frac{B_{max}^2T^2}{\beta}\right)
\label{eq:limiting_descent1}
\end{align}
For the RHS of \eqref{eq:limiting_descent1} to be $\leq 0$, it is necessary that
\begin{align}
    -\alpha_j\epsilon_v\|\mathbb{E}_x[\nabla_{\theta}J_x(\theta_j)]\|^2 + \alpha_j\epsilon_j' + \frac{\alpha_j^2L_JB_{max}^2T^2}{2\beta} &\leq 0 \\
    \frac{\alpha_j^2L_JB_{max}^2T^2}{2\beta} &\leq \alpha_j\epsilon_v\|\mathbb{E}_x[\nabla_{\theta}J_x(\theta_j)]\|^2 - \alpha_j\epsilon_j' \\
    \frac{\alpha_jL_JB_{max}^2T^2}{2\beta} &\leq \epsilon_v\|\mathbb{E}_x[\nabla_{\theta}J_x(\theta_j)]\|^2 - \epsilon_j' \\
    \frac{\alpha_jL_JB_{max}^2T^2}{2\beta} &\leq \epsilon_v\left(\frac{B_{max}^2T^2}{\beta(1-\gamma)^2}\right) - \epsilon_j' \\
    \alpha_j &\leq \frac{2}{L_J}\left(\frac{\epsilon_v}{(1-\gamma)^2} - \frac{\epsilon_j'\beta}{B_{max}^2T^2}\right)
    \label{eq:limiting_step_size}
\end{align}
as was assumed. For the RHS of \eqref{eq:limiting_step_size} to be $> 0$, it is necessary that
\begin{align}
    0 &< \frac{2 \epsilon_v}{L_J(1-\gamma)^2} - \frac{2\epsilon_j'\beta}{L_jB_{max}^2T^2} \\
    \epsilon_j' &< \frac{\epsilon_vB_{max}^2T^2}{\beta(1-\gamma)^2} \\
    \epsilon_j &< \frac{\epsilon_vB_{max}^2}{\beta(1-\gamma)^2}
\end{align}
Since $\epsilon_j' = T^2\epsilon_j$. Thus, descent toward the critical point has been proved for each step $j$. 
Then, using an argument identical to that of \eqref{eq:descent3}-\eqref{eq:descent4}, it follows that, $\forall j$,
\begin{equation}
\mathbb{E}[\mathbb{E}_x[J_x(\theta_{K})]] - \mathbb{E}[\mathbb{E}_x[J_x(\theta_0)]] \leq -\epsilon_v\sum_{j=0}^{K-1}\alpha_j\mathbb{E}[\|\mathbb{E}_{x}[\nabla_{\theta}J_x(\theta_j)]\|^2] +\sum_{j=0}^{K-1}\alpha_j\epsilon_j' + \left(\frac{L_J B_{max}^2 T}{2 \beta}\right)\sum_{j=0}^{K-1}\alpha_j^2
\end{equation}
Because $J_x$ is bounded from below by $J_{inf}$ by Assumption ~\ref{assumption:T}, algebraically rearranging the last line above yields, with $A_K = \sum_{j=0}^{K-1}\alpha_j$,
\begin{align}
&\sum_{j=0}^{K-1}\alpha_j\mathbb{E}[\|\mathbb{E}_{x}[\nabla_{\theta}J_x(\theta_j)]\|^2] \leq 
\frac{\mathbb{E}[\mathbb{E}_x[J_x(\theta_0)]] - J_{inf}}{\epsilon_v} + \frac{1}{\epsilon_v}\sum_{j=0}^{K-1}\alpha_j\epsilon_j' + \frac{L_JB_{max}^2T^2}{2\beta\epsilon_v}\sum_{j=0}^{K-1}\alpha_{j}^2\\
&\frac{1}{A_K}\sum_{j=0}^{K-1}\alpha_j\mathbb{E}[\|\mathbb{E}_{x}[\nabla_{\theta}J_x(\theta_j)]\|^2] \leq \frac{\mathbb{E}[\mathbb{E}_x[J_x(\theta_0)]] - J_{inf}}{\epsilon_vA_K} + \frac{1}{\epsilon_vA_K}\sum_{j=0}^{K-1}\alpha_j\epsilon_j' + \frac{L_JB_{max}^2T}{2\beta\epsilon_vA_K}\sum_{j=0}^{K-1}\alpha_{j}^2
\label{eq:limit3}
\end{align}
Because $\sum_{j=0}^{\infty}\epsilon_j < \infty$ it follows that the partial sums of $\{\epsilon_j'\}_{j=0}^{\infty}$ are bounded. Because $\{\alpha_j\}_{j=0}^{\infty}$ is a decreasing sequence, it follows  by the Dirichlet Test that the product series $\sum_{j=0}^{\infty}\alpha_j\epsilon_1'$ converges. Thus, using linearity of expectation, ~\eqref{eq:limit3} becomes
\begin{equation}
\mathbb{E}\left[\frac{1}{A_K}\sum_{j=0}^{K-1}\alpha_j\|\mathbb{E}_{x}[\nabla_{\theta}J_x(\theta_j)]\|^2\right] \leq \frac{\mathbb{E}[\mathbb{E}_x[J_x(\theta_0)]] - J_{inf}}{\epsilon_vA_K} + \frac{1}{\epsilon_vA_K}\sum_{j=0}^{K-1}\alpha_j\epsilon_j'+ \frac{L_JB_{max}^2T^2}{2\beta\epsilon_vA_K}\sum_{j=0}^{K-1}\alpha_{j}^2
\label{eq:limit4}
\end{equation} 
Hence, since $\lim_{K \rightarrow \infty} A_K = \infty$, we have
\begin{align*}
&\lim_{K \rightarrow \infty}\mathbb{E}\left[\frac{1}{A_K}\sum_{j=0}^{K}\alpha_j\|\mathbb{E}_x[\nabla_{\theta}J_x(\theta_j)]\|^2\right] \\
&\leq \lim_{K \rightarrow \infty}\left[\frac{2\beta(\mathbb{E}[\mathbb{E}_x[J_x(\theta_0)]]-J_{inf}) + 2\beta\sum_{j=0}^{K-1}\alpha_j\epsilon_j'+ L_{J}B_{max}^2T^2{\sum_{j=0}^{K-1}\alpha_{j}^2}}{2 \beta \epsilon_vA_K}\right] \\
&= 0
\end{align*}
\end{proof}
Under the assumptions of Theorem ~\ref{theorem:limiting_theorem}, it is possible to prove Theorem ~\ref{theorem:expectation_liminf} and Corollary ~\ref{corollary:convergence_probability}. These proofs are omitted because they are identical to those already shown.

\end{document}